\documentclass[12pt,a4paper,oneside,english]{amsart}
\usepackage{graphicx}
\usepackage{setspace}
\usepackage{amssymb}
\makeatletter
\usepackage{babel}
 \newtheorem{thm}{Theorem}[section]
 \numberwithin{equation}{section} 
 \numberwithin{figure}{section} 
 \theoremstyle{plain}
  \newtheorem*{thm*}{Theorem}
 \theoremstyle{definition}
 \newtheorem*{defn*}{Definition}
 \theoremstyle{plain}
 \newtheorem{prop}[thm]{Proposition} 
 \theoremstyle{remark}
 \newtheorem{rem}[thm]{Remark}
\theoremstyle{plain}
 \newtheorem{claim}[thm]{Claim}
\theoremstyle{remark}
 \newtheorem{cor}[thm]{Corollary}
 \theoremstyle{plain}
 \newtheorem{lem}[thm]{Lemma} 
 \theoremstyle{definition}
 \newtheorem{defn}[thm]{Definition}
\newtheorem*{acknowledgment}{Acknowledgment}

\AtBeginDocument{
  
}
\begin{document}
\title{``Set-theoretical'' solutions
of the quantum Yang-Baxter equation and a class of  Garside groups.}
\author{Fabienne Chouraqui}

\begin{abstract}
We establish a one-to-one correspondence between structure groups of non-degenerate, involutive  and braided ``set-theoretical'' solutions of the quantum Yang-Baxter equation  and Garside groups with a certain presentation. Moreover,  we show that the solution is indecomposable if  and only if its  structure group is  a $\Delta-$pure Garside group.
\end{abstract}
\maketitle

\section{Introduction}
The quantum Yang-Baxter equation is an equation in the field of mathematical physics and it lies in the foundation of the theory of quantum groups.
Let $R:V  \otimes V \rightarrow V  \otimes V$ be a linear operator, where $V$ is a vector space. The quantum Yang-Baxter equation is the equality $R^{12}R^{13}R^{23}=R^{23}R^{13}R^{12}$ of linear transformations on $V  \otimes V \otimes V$, where $R^{ij}$ means $R$ acting on the $i-$th and $j-$th components.
In this paper, we work with  ``set-theoretical'' solutions of this equation, that is solutions for which $V$ is a vector space spanned by a set $X$ and $R$ is the linear operator induced by a mapping $X \times X \rightarrow X \times X$. The study of these was suggested by Drinfeld \cite{drinf}.
In \cite{etingof}, the authors study ``set-theoretical'' solutions $(X,S)$ of the quantum Yang-Baxter equation satisfying the following conditions: non-degenerate, involutive and braided. To each such solution, they associate a group called the structure group and they show that this group satisfies some properties.
Our work is mostly inspired by the paper of Etingof and al. \cite{etingof} and we use the same notation as their.\\
In this paper, we establish  a one-to-one correspondence between non-degenerate, involutive and braided ``set-theoretical'' solutions of the quantum Yang-Baxter equation (up to isomorphism) and Garside presentations which satisfy some additional conditions \\(up to the t-isomorphism defined below).

The main theorem is the following:
\begin{thm*}\textbf{A}
Let $(X,S)$ be  a non-degenerate, involutive and braided ``set-theoretical'' solution of the quantum Yang-Baxter equation, where $X$ is a finite set. Let $G$ be the structure group corresponding to $(X,S)$.\\
Then $G$ is Garside.\\
Conversely, let $M=\operatorname{Mon} \langle X\mid R \rangle$ be a  Garside monoid such that:\\
1) $X=\{x_{1},..,x_{n}\}$.\\
2) There are $n(n-1)/2$ defining relations in $R$.\\
3) Each side of a relation in $R$ has length 2.\\
4) If the  word $x_{i}x_{j}$ appears in $R$, then it appears only once.\\
 Then there exists a function $S: X \times X \rightarrow X \times X$ such that $(X,S)$ is  a non-degenerate, involutive and braided ``set-theoretical'' solution and $G=\operatorname{Gp} \langle X\mid R \rangle$ is its structure group.
\end{thm*}
The main idea in the proof is that we can express the right and left complement on the generators in terms of the functions which define $(X,S)$.
We define a \textbf{tableau monoid} to be a monoid $M=\operatorname{Mon} \langle X\mid R \rangle$  satisfying the conditions (1) and (3) from Theorem A. The reason of the name is that it can be presented by a table. We say that two tableau monoids $M=\operatorname{Mon} \langle X\mid R \rangle$ and $M'=\operatorname{Mon} \langle X'\mid R' \rangle$ are \textbf{t-isomorphic} if there exists a bijection $s:X \rightarrow X'$ such that $x_{i}x_{j}=x_{k}x_{l}$ is a relation in $R$ if and only if $s(x_{i})s(x_{j})=s(x_{k})s(x_{l})$ is a relation in $R'$. Clearly, if two tableau monoids are t-isomorphic then they are isomorphic and the definition can be enlarged to groups.
We show that if two non-degenerate, involutive and braided ``set-theoretical'' solutions  are isomorphic, then their structure groups (monoids) are t-isomorphic tableau groups (monoids) and conversely t-isomorphic Garside tableau monoids (satisfying additionally the conditions (2) and (4) from Theorem A) yield isomorphic non-degenerate, involutive and braided ``set-theoretical'' solutions.\\
Let $(X,S)$ be  a non-degenerate, involutive and braided ``set-theoretical'' solution of the quantum Yang-Baxter equation, where $X$ is a finite set. Let $G$ be the structure group corresponding to $(X,S)$.\\
 We show that:
 \begin{thm*}\textbf{B}
 $(X,S)$ is indecomposable if and only if $G$ is Garside $\Delta-$pure.
\end{thm*}
\begin{thm*}\textbf{C}
(i) The right least common multiple of the generators is a Garside element. That means that  $G$ is Garside in the sense of \cite{deh_Paris}. \\
(ii) The (co)homological dimension of the structure group $G$ is bounded by the cardinal of $X$.
\end{thm*}
The paper is organized as follows:
In section $2$, we give some preliminaries on ``set-theoretical'' solutions and in section $3$ we give preliminaries on Garside monoids and $\Delta-$pure Garside monoids.
In section $4$, we show that the structure group $G$ of a non-degenerate, involutive and braided ``set-theoretical'' solution is Garside, using the criteria developed by Dehornoy in \cite{deh_francais}.
This implies that $G$ is torsion-free from \cite{deh_torsion} and biautomatic from \cite{deh_francais}.
In section $5$, we show that  the right least common multiple of the generators is a Garside element. That means that  $G$ is Garside in the sense of \cite{deh_Paris}. In section $6$, we show that the (co)homological dimension of the structure group $G$ of a non-degenerate, involutive and braided ``set-theoretical'' solution is bounded by the cardinal of $X$. In section $7$,  we show that a non-degenerate, involutive and braided ``set-theoretical'' solution is indecomposable if and only if its structure group is Garside $\Delta-$pure, using the terminology of Picantin in \cite{picantin}.
In section $8$, we define a tableau monoid (group) to be a monoid (group) such that the relations are quadratic, it can be presented by a table.
We establish the converse implication, that is  a Garside tableau monoid satisfying some additional conditions defines a non-degenerate, involutive and braided ``set-theoretical'' solution of the quantum Yang-Baxter equation.
In section $9$, we consider the special case of  permutation solutions that are not involutive and we show that their structure group is Garside. We could not enlarge this result to general solutions.
In section $10$, we calculate the Garside element for  permutation solutions.
The last section is an appendix which contains the proof of the right  cancellativity of the monoid with the same presentation as the structure group of a non-degenerate, involutive and braided ``set-theoretical'' solution of the quantum Yang-Baxter equation. \\
We remark that independently, using different methods, in  \cite{gateva+van}, the authors define monoids and groups of left and right I-type and they show that they yield solutions to the quantum Yang-Baxter equation. They show also that a monoid of left I-type is cancellative and has a group of fractions that is torsion-free and Abelian-by-finite.
In \cite{jespers+okninski}, the authors extend the results of \cite{gateva+van} and establish a correspondence between groups of I-type and the structure group $G$ of a non-degenerate, involutive and braided ``set-theoretical'' solution.
They also remark that the defining presentation of a monoid of I-type satisfies the right cube condition, as defined by Dehornoy in \cite[Prop.4.4]{deh_complte}. Using our result, this makes a correspondence between groups of I-type and the class of Garside groups  studied in this paper.
Gateva-Ivanova shows in \cite{gateva} that the monoid corresponding to a square-free,  non-degenerate, involutive and braided ``set-theoretical'' solution has a structure of distributive lattice with respect to left and right divisibility and that the left least common multiple of the generators is equal to their right least common multiple and she calls this element the principal monomial.

\begin{acknowledgment}
This work is a part of   my Ph-d thesis, done  at the Technion under
the supervision of Professor Arye Juhasz.  I am very grateful to
Professor Arye Juhasz, for his patience, his encouragement and his many helpful remarks.
\end{acknowledgment}
\pagebreak
\section{The structure group corresponding to a ``set-theoretical'' solution}

All the definitions and results in this section are from \cite{etingof}.

\subsection{Definitions}
\begin{defn}
Let $X$ be a non-empty set and $S:X^{2}\rightarrow  X^{2}$ be a bijection. We will denote the components of $S$  by $S_{1}$ and $S_{2}$ (i.e $S(x,y)=(S_{1}(x,y),S_{2}(x,y))$); they are binary operations on $X$. For positive integers $i<n $ let the map $S^{ii+1}: X^{n} \rightarrow X^{n}$ be defined by $S^{ii+1}=id_{X^{i-1}} \times S\times id_{X^{n-i-1}} $.\\
(i) A pair $(X,S)$ is called \textbf{nondegenerate} if the maps $X \rightarrow X$ defined by $x \mapsto S_{2}(x,y)$ and $x \mapsto S_{1}(z,x)$ are bijections for any fixed $y,z \in X$.
(ii) A pair $(X,S)$ is called \textbf{braided} if $S$ satisfies the braid relation
$S^{12}S^{23}S^{12}=S^{23}S^{12}S^{23}$.
(iii) A pair $(X,S)$ is called \textbf{involutive} if $S^{2}=id_{X^{2}}$, that is $S^{2}(x,y)=(x,y)$ for all $x,y \in X$.
A braided set which is involutive is called a \textbf{symmetric set}.\\
(iiii) Pairs $(X,S)$ and  $(X',S')$  are said to be \textbf{isomorphic} if there exists a bijection $\phi: X \rightarrow X'$ which maps $S$ to $S'$, that is $S'(\phi(x),\phi(y))=(\phi(S_{1}(x,y)),\phi(S_{2}(x,y)))$.
\end{defn}
Let $\alpha:X \times X \rightarrow X\times X$ be the permutation map, defined by $\alpha(x,y)=(y,x)$. Let $R=\alpha \circ S$. The map $R$ is called the \textbf{$R-$matrix corresponding to $S$}.
\begin{prop}\cite[Prop.1.2]{etingof}
(i) $(X,S)$ is a braided set if and only
 if $R$ satisfies the quantum Yang-Baxter equation $R^{12}R^{13}R^{23}=R^{23}R^{13}R^{12}$, where
$R^{ij}$ means acting on the $i$th and $j$th components.\\
(ii) $(X,S)$ is a symmetric  set if and only if in addition $R$ satisfies the unitary condition $R^{21}R=1$.
\end{prop}
\textbf{Example:}
Let $X$ be a set and let  $S:X^{2} \rightarrow X^{2}$  be the  mapping  $S(x,y)=(y,x)$. Then $(X,S)$ is a nondegenerate symmetric set which is called the \textbf{trivial solution}. \\
\textbf{Another example:the permutation solution  (Lyubashenko)}\\
If $S(x,y)=(g(y),f(x))$, where $f,g:X\rightarrow X$ then $(X,S)$ is  nondegenerate iff
$f,g$ are bijective; $(X,S)$ is  braided  iff $fg=gf$; $(X,S)$ is
involutive  iff $g=f^{-1}$. In the last case $(X,S)$ is   called a \textbf{permutation solution}.  Two permutation solutions are isomorphic if and only if  the corresponding permutations are conjugate.

 The notation introduced in \cite{etingof} is as
follows: Let $X$ be  a finite set and let $S$ be defined in the
following way: $S(x,y)=(g_{x}(y),f_{y}(x))$.
 Here, if $X=\{x_{1},...,x_{n}\}$ is a finite set and $y=x_{i}$ for some $1\leq i\leq n$, then
 we write sometimes $f_{i},g_{i}$ instead of $f_{y},g_{y}$  and we use the following notation for $S$: $S(i,j)=(g_{i}(j),f_{j}(i))$.\\
The following claim is Proposition 1.6 from \cite{etingof} with some additions which are implicit from the proof of the proposition, and which is adapted to our needs.
\begin{claim}\label{debut_form}
(i) $S$ is non-degenerate $\Leftrightarrow$ $f_{i}$ and $g_{i}$ are bijective, $1 \leq i \leq n$.\\
(ii) $S$ is involutive $\Leftrightarrow$ $g_{g_{i}(j)}f_{j}(i)=i$  and $f_{f_{j}(i)}g_{i}(j)=j$,
$1 \leq i,j \leq n$.\\
(iii) $S$ is braided  $\Leftrightarrow$ $g_{i}g_{j}=g_{g_{i}(j)}g_{f_{j}(i)}$  and $f_{j}f_{i}=f_{f_{j}(i)}f_{g_{i}(j)}$\\ and $f_{g_{f_{j}(i)}(k)}g_{i}(j)= g_{f_{g_{j}(k)}(i)}f_{k}(j)$, $1 \leq i,j,k \leq n$.
\end{claim}

\subsection{The structure group}
\begin{defn}\label{def_rel_struct}
The \textbf{structure group of $(X,S)$}  is defined to be the group, $G$,  generated by the elements of $X$ with defining relations \
$xy=tz$ when $S(x,y)=(t,z)$.\
\end{defn}

\begin{prop}\label{braided_actions}\cite[Prop.2.1]{etingof}
Suppose that $(X,S)$ is non-degenerate. Then $(X,S)$ is a braided set if and only if the following conditions are simultaneously satisfied:\\
(i) the assignment $x \rightarrow f_{x}$ is a right action of $G$ on $X$.\\
(ii) the assignment $x \rightarrow g_{x}$ is a left action of $G$ on $X$.\\
(iii) the linking relation $f_{g_{f_{y}(x)}(z)}(g_{x}(y))=g_{f_{g_{y}(z)}(x)}(f_{z}(y))$ holds.
\end{prop}
\begin{prop}\label{Prop_def_T}\cite[Prop.2.2]{etingof}
Suppose that $(X,S)$ is non-degenerate, involutive and braided.
Define the map $T:X \rightarrow X$ by the formula $T(y)=f^{-1}_{y}(y)$.
Then\\
  (i) $Tg_{y}=f^{-1}_{y}T$\\
  (ii) $T$ is invertible and $T^{-1}(y)=g^{-1}_{y}(y)$.
\end{prop}
\begin{defn}\label{def_decomposale}\cite[Defn.2.5]{etingof}
(a) A subset $Y$ of a non-degenerate and symmetric set $X$ is said to be an \textbf{invariant} subset if $S(Y \times Y)\subseteq Y \times Y$.\\
(b) An invariant subset $Y$  is said to be  \textbf{non-degenerate} if $(Y,S\mid_{Y\times Y})$ is a non-degenerate and symmetric set.\\
(c) A non-degenerate and symmetric set $(X,S)$ is said to be \textbf{decomposable} if it is a union of two nonempty disjoint non-degenerate invariant subsets. Otherwise,  $(X,S)$ is said to be \textbf{indecomposable}.
\end{defn}
\begin{rem}\cite{etingof}
If $X$ is finite, then any invariant subset $Y$ of $X$ is non-degenerate.
\end{rem}
\begin{prop}\cite[Prop.2.11]{etingof}\label{pro_etin_Gtransitiv}
A non-degenerate and symmetric set $(X,S)$ is indecomposable if and only if $G$ acts transitively on $X$.
\end{prop}
\begin{prop}\cite[Prop.2.15]{etingof}\label{prop_eting_YZ}
If $(X,S)$ is a union of non-degenerate invariant subsets $Y$ and $Z$, then the map $S$ defines bijections $Y \times Z \rightarrow Z \times Y$ and $Z \times Y \rightarrow Y \times Z$.
\end{prop}
\pagebreak
\section{ Garside monoids and groups}
All the definitions and results in the first two  subsections are taken from
\cite{deh_francais}, \cite{deh_livre} and in the third subsection also from \cite{picantin}.
\subsection{Definitions}
\begin{defn}
Let $M$ be a monoid and let $x,y$ be elements in $M$. Call $x$
\textbf{a left divisor} of $z$ if there is an element $t$ in $M$
such that $z=xt$. Call $x$ is \textbf{a proper} left divisor of $z$ if in addition
 $t \neq 1$.
\end{defn}
\begin{defn}
Let $M$ be a monoid and let $x,y$ be elements in $M$. Say that $z$
is
\textbf{a right least common multiple  (right lcm)} of $x$ and $y$ if:\\
  $1.$ $x$ and $y$ are left divisors of $z$ \\
  $2.$ if there is an element $w$ in $M$ such that $x$ and $y$ are left divisors of $w$, then
   $z$ is  left divisor of $w$.
   The notation is $z=x\vee y$.\\
  \textbf{The complement at right of
   $y$ on $x$} is defined to be an element $c\in M$ such that $x\vee y= xc$. Notation:
   $c=x \setminus y$. By definition, $x\vee y=x(x \setminus y)$.
\end{defn}

\begin{defn}
Let $M$ be a monoid and let $z,w$ be elements in $M$. Say that $u$
is
\textbf{a left greatest common divisor  (left gcd)} of $x$ and $y$ if:\\
  $1.$  $x$ is a  left divisor of $z$ and $w$. \\
  $2.$ if there is an element $v$ in $M$ such that $v$ is a left divisor of $z$ and $w$, then
   $v$ is  a left divisor of $x$.
   The notation is $x=z \wedge w$.
\end{defn}
The left lcm and the right gcd of two elements are defined in a
symmetric way.
 If the element $x$ of $M$ is  the equivalence class of the word $w$, we say that \textbf{$w$ represents $x$}.
\begin{defn}
Let $M$ be a monoid and let $x$ be an element in $M$. \\
(i) Call  $x$ \textbf{an atom}  if $x$ is not equal to 1 and  $x=yz$ implies
that $y=1$ or $z=1$. \\
(ii)  The \textbf{norm} $\parallel x \parallel$ of
$x$ is defined to be the supremum of the lengths of the
decompositions of $x$ as a product of atoms.\\
(iii) The monoid  $M$ is
\textbf{atomic} if $M$ is generated by its atoms and for every $x$
in $M$ the norm of $x$ is finite.
\end{defn}
\textbf{Example:}\cite{deh_francais} Let $M$ be the monoid presented by  $ M = \langle a,b
\mid aba=baab\rangle  $.  The word $abaa$ has infinite norm since it holds  that
$abaa=baaba=babaab=...$, so $M$ is not atomic.\\
\textbf{Basic Facts}\\
(a) If all the relations in $M$
are length preserving, then each element $x$ of $M$ has a finite norm, since all the words which represent $x$ have the same length. So,  $M$ is atomic.\\
(b) If $M$ is atomic then it holds  that $\parallel x \parallel \geq 1$ for
$x \neq 1 $ and  $\parallel xy\parallel \geq\parallel x
\parallel + \parallel y \parallel$ for every $x,y$ in $M$.\\
(c) If $M$ is atomic then $1$ is the unique invertible
element (from (b)).
\begin{defn}
Let $M$ be a monoid.\\
(a) $M$ is said to be \textbf{Gaussian} if $M$ is
atomic, left and right cancellative, and if any two elements in
$M$ have a left and right gcd and a left and right lcm.\\
(b) Let $\Delta \in M$. Then $\Delta$ is a \textbf{Garside element} if the left divisors of
$\Delta$ are the same as the right divisors, there is a finite
number of them and they generate $M$.\\
(c) A Gaussian monoid $M$ is said to be   \textbf{Garside} if $M$ contains a Garside element.\\
(d) A group $G$ is said to be a \textbf{Gaussian group} (respectively
a \textbf{Garside group}) if there exists a Gaussian monoid $M$
(respectively a Garside monoid) such that $G$ is the fraction
group of $M$. Clearly, if $G$ is Garside, then $G$ is Gaussian.
\end{defn}
A Gaussian monoid satisfies both left and right Ore's conditions, so it embeds in its group of fractions (see \cite{Clifford}).\\
\textbf{Example} Braid groups \cite{garside} and Arting groups of finite type are Garside
groups.
\begin{defn}\cite[Defn.1.6]{deh_francais}\label{def_conditions}
Let $M$ be a monoid. $M$ satisfies the condition:\\
 $-(C_{0})$ if $1$ is the unique invertible element in $M$.\\
 $-(C_{1})$ if $M$ is left cancellative.\\
$-(\tilde{C_{1}})$ if $M$ is right cancellative.\\
$-(C_{2})$ if any two elements in $M$ with a right common multiple
admit
a right lcm.\\
$-(C_{3})$ if  $M$ has a finite generating set $P$ closed under
$\setminus$, i.e if $x,y \in P$ then $x \setminus y \in P$.
\end{defn}
\begin{thm}\cite{deh_francais}[Prop.2.1] \label{gars_critere}
A monoid $M$ is a Garside monoid if and only if $M$ satisfies the
conditions $(C_{0})$ , $(C_{1})$ , $(\tilde{C_{1}})$ , $(C_{2})$
and  $(C_{3})$.
\end{thm}
\subsection{Recognizing Garside monoids}

\begin{defn}
\cite[Defn.4.2]{deh_francais}\cite[p.49]{deh_livre} Let $X$ be an alphabet and let
denote by $\epsilon$ the empty word in $X^{*}$. Let $f$ be  a partial function of $X\times X$ into $X^{*}$. We say that $f$ is
a \textbf{complement} on $X$ if  $f(x,x)=\epsilon$ holds for every
$x$ in $X$, and $f(y,x)$ exists whenever $f(x,y)$ does. \\
We write  ``$\equiv^{+}$''
for the congruence on  $X^{*}$ generated by the pairs\\
$(xf(x,y),yf(y,x))$ with $(x,y)$ in the domain of $f$, and we write
``$\equiv$'' for the congruence on $(X\cup X^{-1})^{*}$ generated by
$\equiv^{+}$ together with all  pairs $(xx^{-1},\epsilon)$ and
$(x^{-1}x,\epsilon)$ with $x\in X$. We define the monoid and the
group  \textbf{associated with $f$} to be $X^{*}/ \equiv^{+}$ and
$(X\cup X^{-1})^{*}/ \equiv$ respectively, or in other words the monoid $\operatorname{Mon}\langle    X\mid xf(x,y)=yf(y,x)\rangle$ and the group $\operatorname{Gp}\langle    X\mid xf(x,y)=yf(y,x)\rangle$ respectively
(with $(x,y)$ in the domain of $f$).
\end{defn}
\textbf{Example $1$:} Let$X=\{x_{1},x_{2},x_{3},x_{4},x_{5}\}$ and
$S(i,j)=(g_{i}(j),f_{j}(i))$.\\
We use the following notation:
$\sigma=(1,2,3,4)(5)$ means that $\sigma(1)=2$,
$\sigma(2)=3$, $\sigma(3)=4$, $\sigma(4)=1$ and $\sigma(5)=5$.\\
Assume $\begin{array}{c}
f_{1}=g_{1}=(1,2,3,4)(5)\\
f_{2}=g_{2}=(1,4,3,2)(5)\\
f_{3}=g_{3}=(1,2,3,4)(5)\\
f_{4}=g_{4}=(1,4,3,2)(5)\\
f_{5}=g_{5}=(1)(2)(3)(4)(5)\\
\end{array}$\\
Then a case by case checking  shows that  $(X,S)$ is a non-degenerate, involutive and braided solution.\\It follows from Definition \ref {def_rel_struct}, that the
following $10$  relations define the corresponding structure
group: $\begin{array}{ccc}
x^{2}_{1}=x^{2}_{2}   && x^{2}_{3}=x^{2}_{4}\\
x_{1}x_{2}=x_{3}x_{4} && x_{1}x_{5}=x_{5}x_{1}\\
x_{1}x_{3}=x_{4}x_{2} && x_{2}x_{5}=x_{5}x_{2}\\
x_{2}x_{4}=x_{3}x_{1} && x_{3}x_{5}=x_{5}x_{3}\\
x_{2}x_{1}=x_{4}x_{3} && x_{4}x_{5}=x_{5}x_{4}\\
\end{array}$\\
The relations which are omitted are the trivial ones, e.g.  $x_{1}x_{4}=x_{1}x_{4}$ and so on.
We consider the monoid $M$ with the same presentation. The
complement $f$ is defined totally on $X \times X$.
As an example,   $f(x_{1},x_{2})=x_{1}$  and
$f(x_{2},x_{1})=x_{2}$ are obtained from the relation
$x^{2}_{1}=x^{2}_{2}$,  since it holds that $f(x_{1},x_{2})=x_{1}\setminus
x_{2}$ and so on.  Note that the monoid associated to $f$, $X^{*}/
\equiv^{+}$, is  the monoid $M$.\\
The complement mapping considered sofar is defined on letters
only. Its extension on words is called \textbf{word reversing}  and
is defined in the following way \cite{deh_livre}:\\
If $f$ is a complement on $X$ , then by definition it holds that\
$xf(x,y)\equiv^{+}yf(y,x)$ for all $x,y \in X$ and this implies
that \ $x^{-1}y\equiv f(x,y)f(y,x)^{-1}$. Thus, if we replace in a
word a pattern(subword)  of the form  $x^{-1}y$ with the
corresponding pattern  $f(x,y)f(y,x)^{-1}$, we obtain an
equivalent word. \textbf{Reversing} a word will consist in
iterating this transformation so as to eventually obtain a word where all positive letters occur before all negative letters.
\begin{defn}\cite[Defn.1.13]{deh_livre}
Let $f$ be as above and let $u,v$ be words in $X^{*}$. Then the word $u\setminus v$ is
defined to be the word $v_{1}$ such that $u^{-1}v$ reverses
to $v_{1}u_{1}^{-1}$ for some word $u_{1}$, if such words exist. If $u_{1}$, $v_{1}$  exist, then  they are unique from the definition of $f$.
\end{defn}
\begin{defn} \cite[Defn.5.3]{deh_francais}\cite[p.59]{deh_livre}
Let $f$ be a complement on $X$. For $u,v,w \in X^{*}$ we say that
$f$ is \textbf{coherent} at $(u,v,w)$ if either \\ $( (u
\setminus v ) \setminus (u \setminus w))
\setminus((v \setminus u) \setminus(v \setminus
w))\equiv^{+} \epsilon$ \ holds, or neither of the words $( (u
\setminus v )\setminus (u\setminus w)) ,    ((v
\setminus u)\setminus(v \setminus w))$ exists. For $Y
\subseteq X^{*}$, we say that \textbf{$f$ is coherent on $Y$} if it is
coherent at every triple $(u,v,w)$  with $u,v,w \in Y$. We say
that \textbf{$f$ is coherent} if it is coherent on all of $X^{*}$.
 \end{defn}
\begin{prop} \cite[Prop.6.1]{deh_francais}\label{atomic_coh}
Let $M$ be a monoid associated with a complement $f$ and assume
that $M$ is atomic. Then $f$ is coherent if and only if $f$ is
coherent on $X$.
\end{prop}
\textbf{Example:} In  example $1$, the monoid is atomic since all
the defining relations are length-preserving. So, due to proposition \ref{atomic_coh}, in order to
check the coherence of $f$ it is enough to check the coherence on
$X$ only. As an example,we will check if
$( (x_{1} \setminus
x_{2} )\setminus (x_{1}\setminus x_{3}))
\setminus((x_{2} \setminus x_{1})\setminus(x_{2}
\setminus x_{3}))= \epsilon$ holds in $M$:\\ $\begin{array}{ccc}
 x_{1}\setminus x_{2}=x_{1} & x_{2}\setminus x_{1}=x_{2}\\
x_{1}\setminus x_{3}=x_{2} & x_{2}\setminus x_{3}=x_{4}\\
(x_{1}\setminus x_{2})\setminus(x_{1}\setminus x_{3})=x_{1}\setminus x_{2}=x_{1}
&
(x_{2}\setminus x_{1})\setminus(x_{2}\setminus x_{3})=
x_{2}\setminus x_{4}=x_{1}\\
\end{array}$
\\
So,
$((x_{1}\setminus x_{2})\setminus(x_{1}\setminus x_{3}))
\setminus
((x_{2}\setminus x_{1})\setminus(x_{2}\setminus x_{3}))=
x_{1}\setminus x_{1}=\epsilon $
\begin{prop} \cite[Lemmas.5.7-5.9]{deh_francais} \cite[p.55]{deh_livre}\label{coh_c1,2}
Let $M$ be a monoid associated with a coherent complement. Then
$M$ satisfies $C_{1}$  and $C_{2}$ (i.e
 $M$ is left
cancellative and any two elements in $M$ with a right common
multiple admit a right lcm respectively).
\end{prop}
\subsection{$\Delta-$pure Garside monoids and groups}

\textbf{Notation:} Assume that $M$ is a Gaussian monoid. For $X,Y \subseteq M$, we denote by $Y \setminus X$ the set of elements $b \setminus a$ for $a \in X$ and $b \in Y$. We write $Y \setminus a$ for $Y \setminus \{a\}$ and $b \setminus X$ for $\{b\} \setminus X$.
\begin{defn}\cite[Defn.1.11]{deh_francais}
Let $M$ be a Garside monoid with  $X$ a set of atoms. An element
is called \textbf{simple} if it belongs to the closure of $X$
under right complement and right lcm.
\end{defn}
\textbf{Example:}\cite[Ex.1.1]{picantin}
$M=\operatorname{Mon}\langle x,y \mid xyyxyxyyx=yxyyxy \rangle$ is a Garside monoid.
As an example, the word $yyxyxyyx$ which represents the element $x \setminus y$ is a simple element.
\begin{defn}\cite[p.95]{picantin}
Assume that $M$ is a Garside monoid, $\chi$ is its set of simples and $\Delta$ its Garside element. The \textbf{exponent} of $M$ is the order of the automorphism $\phi$, where $\phi$ is the extension of the function $x \rightarrow (x\setminus \Delta)\setminus \Delta$ from $\chi$ into itself.
\end{defn}
\begin{defn}\cite[p.97]{picantin}
Assume that $M$ is a Garside monoid. For every $a$ in $M$, we define
$\Delta_{a} = \vee \{b \setminus a ; b \in M\}$, where $\vee$ denotes the right lcm.
\end{defn}
\begin{defn}\cite[p.116]{picantin}
Assume that $M$ is a Garside monoid and let $X$ be its set of atoms. The monoid $M$ is said to be \textbf{$\Delta-$pure} if for every $x,y$ in $X$, it holds that $\Delta_{x} = \Delta_{y}$.
\end{defn}
\begin{prop}\cite[Prop.4.1]{picantin}\label{center_cyclic}
Assume that $M$ is a $\Delta-$pure Garside monoid, $\Delta$ is its Garside element, $e$ is its exponent and $G$ its group of fractions.
Then the center of $M$ (resp. of $G$) is the infinite cyclic submonoid (resp. subgroup) generated by $\Delta^{e}$.
\end{prop}
In \cite{picantin}, Picantin defines a monoidal version of the definition of crossed product and he shows the following:
\begin{prop}\cite[Prop.4.5]{picantin}
A Garside monoid is an iterated crossed product of $\Delta-$pure Garside submonoids.
\end{prop}
\section{Structure groups of non-degenerate and symmetric ``set-theoretical'' solutions are Garside}
As before, we have  $X=\{x_{1},x_{2},...,x_{n}\}$ and
$S(i,j)=(g_{i}(j),f_{j}(i))$.\\
We recall that the relations are obtained in the following
way:\\$S(i,j)=(k,l)$ implies $x_{i}x_{j}=x_{k}x_{l}$
for every $x_{i},x_{j} \in X$. \\The structure group corresponding
to $(X ,S)$ is denoted by  $G$ and $M$ is the monoid with the same
presentation. The aim of this section is to prove the following
theorem:
\begin{thm}\label{theo:garside}
The structure group $G$ of a non-degenerate, braided and involutive
``set-theoretical'' solution of the quantum Yang-Baxter equation is a Garside group.
\end{thm}
In order to prove that the group $G$ is a Garside group, we will
 show that \textbf{the monoid $M$ with the same presentation} is
a Garside monoid. For that, we will use the Garsidity criterion
given in  theorem \ref{gars_critere}, that is we will show that
$M$ satisfies the conditions $(C_{0})$ , $(C_{1})$ , $(C_{2})$,
$(C_{3})$ and  $(\tilde{C_{1}})$.
\subsection{$M$  is atomic}

We will show that $M$ is atomic and this implies that $M$ satisfies $(C_{0})$, that is $1$ is the unique invertible element in $M$.
In order to show that $M$ is atomic, we describe the relations and show that the relations are length-preserving.
\begin{claim}\label{cl_compl}
Assume $(X,S)$ is non-degenerate.
Let $x_{i}$ and $x_{j}$ be different elements in $X$ (i.e generators of $M$). Then
there is exactly one defining relation $x_{i} a= x_{j} b$, where
$a,b$ are in $X$. Hence, $f$ is defined in $X \times X$.
If in addition, $(X,S)$ is involutive then $a$ and $b$ are different.
\end{claim}
\begin{proof}
We have to show that for each pair  of generators $x_{i}$ and $x_{j}$,
 there are two different generators $a$ and $b$ such that we have
$S(i,a)=(j,b)$.
We recall that $S$ is non-degenerate implies that the functions
 $g_{k}$ are bijective for every $1\leq k \leq n$.
So, let us take  $a$ to be $a=g_{i}^{-1}(j)$. That  $a$ exists and is unique follows from the injectivity of $g_{i}$.
Next, take $b$ to be $b=f_{a}(i)$.\\
So, it holds that $S(i,a)=(g_{i}(a),f_{a}(i))=(j,b)$.\\
Assume there are two defining relations $x_{i} a= x_{j} b$ and $x_{i} c= x_{j} d$, where
$a,b,c,d$ are in $X$ and $a \neq c$. Then we obtain $g_{i}(a)=j$ and $g_{i}(c)=j$ which contradicts the fact that $g_{i}$ is bijective, so the defining  relation $x_{i} a= x_{j} b$ is unique.\\
Assume  $a=b$, that is $S(i,a)=(j,a)$ . Since $S$ is involutive, we have that $S(i,a)=(g_{i}(a),f_{a}(i))=(j,a)$ and
$S(j,a)=(g_{j}(a),f_{a}(j))=(i,a)$, that is $f_{a}(i)=f_{a}(j)=a$.
But this contradicts the injectivity of $f_{a}$, so $a,b$ are different.
\end{proof}
\begin{rem}\label{rem_cancel}
From claim \ref{cl_compl}, if $(X,S)$ is non-degenerate and involutive then  there can be no relation of the form $x_{i}a=x_{j}a$, where $i \neq j$. Using the same arguments (involving the injectivity of the functions $g_{.}$),  there can be no relation of the form $ax_{i}=ax_{j}$, where $i \neq j$ (see lemma \ref{lem_cancel}).
\end{rem}
\begin{claim} \label{complement_f_defined}
Assume $(X,S)$ is non-degenerate.\\
(a)The complement $f$ is totally defined on $X \times X$,
its range is $X$ and the monoid associated to $f$ is $M$.\\
(b) $M$ is atomic.
\end{claim}
\begin{proof}
(a)This is a direct result from claim \ref{cl_compl}.
 The complement $f$ is defined in a unique way and the
  congruence generated by the relations $xf(x,y)=yf(y,x)$ is exactly the same
   as the congruence generated by the relations in $M$, so the monoid associated to $f$ is $M$.\\
(b)    $M$ is atomic since
 all the defining relations are length-preserving.
\end{proof}
\begin{cor}\label{M_atomic}
Assume $(X,S)$ is non-degenerate and involutive.
 There are $4$ kinds of  defining relations which can be described
in the following way:
\begin{tabular}{|c|c|c|}\hline
  $x^{2}_{i}=x_{j}x_{k}$ & quad &  $j,k \neq i$ \\
  $x_{i}x_{j}=x_{k}x_{l}$ &4-diff  & $\{k,l\} \neq \{i,j\}$ \\
  $x_{i}x_{j}=x_{j}x_{k}$   & log & $k \neq i$ \\
  $x_{i}x_{j}=x_{j}x_{i}$ & comm &\\ \hline
\end{tabular}\\
Note that if $(X,S)$ is not involutive, then  relations of the form  $x_{i}x_{j}=x_{i}x_{k}$  and $x_{j}x_{i}=x_{k}x_{i}$, where $j \neq k$,  occur.
\end{cor}
\begin{proof}
(a) These $4$ kinds of rules correspond respectively to the following
four cases:\\
  $S(i,i)=(g_{i}(i),f_{i}(i))=(j,k)$,
  where $j,k \neq i$. \\
  $S(i,j)=(g_{i}(j),f_{j}(i))=(k,l)$,
  where $\{k,l\} \neq \{i,j\}$. \\
  $S(i,j)=(g_{i}(j),f_{j}(i))=(j,k)$, where $k \neq i$
    and $x_{j}$ is a fixed point of $g_{i}$. \\
  $S(i,j)=(g_{i}(j),f_{j}(i))=(j,i)$, where  $x_{j}$ is a
   fixed point of $g_{i}$ and  $x_{i}$ is a fixed point of $f_{j}$. \\
\end{proof}
\subsection{$M$ satisfies the conditions $(C_{1})$ and $(C_{2})$}

From claim \ref{complement_f_defined}, we have that there is a one-to-one correspondence  between the complement $f$ and the monoid $M$ with the same presentation as the structure group, so we will say that $M$ is coherent (by abuse of notation).\\
In order to show that the monoid $M$  satisfies the conditions
$(C_{1})$ and $(C_{2})$, we will show that $M$ is coherent (from Proposition \ref{coh_c1,2}).
Now, since $M$ is atomic, it is enough to check its coherence on $X$
(from Corollary \ref{M_atomic} and Proposition \ref{atomic_coh}).
 So, we will show that  any triple of generators $(a,b,c)$   satisfies the following equation:\\
 $( (a \setminus b )\setminus (a\setminus c))
\setminus((b \setminus a)\setminus(b \setminus
c))\equiv^{+} \epsilon$.\
 In fact, we show that  any triple of generators $(a,b,c)$ satisfies the following equation:\\
 $ (a \setminus b )\setminus (a\setminus c)=
(b \setminus a)\setminus(b \setminus c)$,\\ that is the equality is in the free monoid $X^{*}$,
 since  the range of $f$  is $X$.\\
In order to make the computations easier, we will prove the
following lemmas.
\begin{lem}\label{form_compl}
Assume $(X,S)$ is non-degenerate.\\
Let $x_{i}, x_{j}$ be different elements in $X$.\\ Then
$x_{i}\setminus x_{j}= g^{-1}_{i}(j)$.
Note that $x_{i}\setminus x_{i}= \epsilon$, where $\epsilon$ denotes the empty word.
\end{lem}
\begin{proof}
If $S(i,a)=(j,b)$, then $x_{i}\setminus x_{j}=a$.
But by definition of $S$, we have that $S(i,a)=(g_{i}(a),f_{a}(i))$,
 so it holds that $g_{i}(a)=j$
 which gives $a=g_{i}^{-1}(j)$.
\end{proof}
\begin{lem}\label{formule}
Assume $(X,S)$ is non-degenerate and involutive.\\
Let $x_{i}, x_{k}$ be elements in $X$.\\ Then
$g^{-1}_{k}(i)=f_{g^{-1}_{i}(k)}(i)$
\end{lem}
\begin{proof}
$S$ is involutive so we have from claim \ref{debut_form} that for every $x_{i}, x_{j} \in X$, $g_{g_{i}(j)}f_{j}(i)=i$.\\
Let replace in this formula  $j$  by $g^{-1}_{i}(k)$ for some $1 \leq k \leq n$, then we obtain
$i= g_{g_{i}(g^{-1}_{i}(k))}f_{g^{-1}_{i}(k)}(i)=g_{k}f_{g^{-1}_{i}(k)}(i)$\\
So, we have $g^{-1}_{k}(i)=f_{g^{-1}_{i}(k)}(i)$.
\end{proof}
\begin{prop}\label{M_c1&c2}
Assume $(X,S)$ is non-degenerate, involutive and braided. Every triple  $(x_{i},x_{k},x_{m})$ of generators satisfies the following equation: $ (x_{i} \setminusýx_{k} )\setminus (x_{i}\setminus x_{m})
=(x_{k} \setminus x_{i})\setminus(x_{k} \setminus
x_{m})$.\\
$M$ is coherent and satisfies the conditions $(C_{1})$ and $(C_{2})$.
\end{prop}
\begin{proof}
If $x_{i}=x_{k}$ or $x_{i}=x_{m}$ or $x_{k}=x_{m}$, then the equality holds trivially. So, assume that  $(x_{i},x_{k},x_{m})$ is a triple of different generators. This implies that $g_{i}^{-1} (k) \neq g_{i}^{-1} (m)$ and $g^{-1}_{k}(i) \neq g^{-1}_{k}(m)$, since the functions $g_{i}$ are bijective.\\
So, from lemma \ref{form_compl}, we have the following formulas for all different $1 \leq i,k,m \leq n$:\\
 $ (x_{i} \setminus x_{k} )\setminus (x_{i}\setminus x_{m})=
  g^{-1}_{x_{i} \setminus x_{k}} (x_{i}\setminus x_{m})=
   g^{-1}_{g_{i}^{-1}(k)} g_{i}^{-1} (m)$\\
$ (x_{k} \setminus x_{i} )\setminus (x_{k}\setminus
x_{m})=
 g^{-1}_{x_{k} \setminus x_{i}} (x_{k}\setminus x_{m})=
 g^{-1}_{g_{k}^{-1}(i)} g_{k}^{-1} (m)$\\
So, we have to show for all different $1 \leq i,k,m \leq n$ that: \\
 $g^{-1}_{g_{i}^{-1}(k)} g_{i}^{-1} (m)=g^{-1}_{g_{k}^{-1}(i)} g_{k}^{-1} (m)$\\
 We  show  for all  $1 \leq i,k \leq n$ that  $g_{i}g_{g_{i}^{-1}(k)}= g_{k}g_{g_{k}^{-1}(i)}$ and this  will imply the above equality.\\
 $S$ is braided, so from claim \ref{debut_form}, we have that:\\
 $g_{i}g_{g_{i}^{-1}(k)}=g_{g_{i}(g_{i}^{-1}(k))}g_{f_{g_{i}^{-1}(k)}(i)}
 =$ $ g_{k}g_{f_{g_{i}^{-1}(k)}(i)}$.\\
 But, from lemma \ref{formule}, $f_{g^{-1}_{i}(k)}(i)=g^{-1}_{k}(i)$, so
 $g_{i}g_{g_{i}^{-1}(k)}= g_{k}g_{g^{-1}_{k}(i)}$.\\
 So, $M$ is coherent at $X$ but since $M$ is atomic we have that $M$ is
coherent. So, $M$  satisfies the conditions $(C_{1})$ and $(C_{2})$.
\end{proof}

\subsection{$M$ satisfies the conditions $(C_{3})$}
\begin{claim}
Assume $(X,S)$ is non-degenerate, involutive and braided.\\
There is a finite generating set which is closed under
$\setminus$, i.e $M$ satisfies the condition $(C_{3})$.
\end{claim}
\begin{proof}
 From
claim \ref{cl_compl}, for any pair of generators $x_{i}, x_{j} $ there are unique $a,b \in X$ such that  $x_{i} a= x_{j} b$, that is   any pair of generators $x_{i}, x_{j} $ has
a right common multiple. Since from claim
\ref{M_c1&c2}, $M$ satisfies the condition $(C_{2})$,  we have
that $x_{i}$ and $x_{j}$ have a right lcm and the word $x_{i} a$ (or
$ x_{j} b$) represents the element $x_{i} \vee x_{j}$, since this is a common multiple of $x_{i}$ and
$x_{j}$ of least length. So,
 it holds that $x_{i}  \setminus x_{j}=a$ and $ x_{j}
\setminus x_{i} = b$, where $a,b \in X$. So,  $X \cup \{\epsilon\}$ is
closed under $\setminus$.
\end{proof}
\subsection{$M$ is right cancellative, i.e  $M$  satisfies $(\tilde{C_{1}})$}
The following lemma is useful for the proof of Claim \ref{M_right_cancel} and also for the calculations in Section 8.
\begin{lem}
\label{lem_cancel}
Assume  $(X,S)$ is non-degenerate and involutive.
Let $a$ and $a'$ be different elements in $X$ (i.e generators of $M$). Then
there is exactly one defining relation $x_{i} a= x_{j} a'$, where
$x_{i}, x_{j}$ are in $X$ and are different.
\end{lem}
\begin{proof}
We have to show that for each pair $a,a'$ of generators ,
 there are two generators $x_{i}$ and $x_{j}$ such that we have
$S(i,a)=(j,a')$.
We recall that $S$ is non-degenerate which implies that the functions  $f_{a}$
are bijective for every $a \in X$.
So, let take  $x_{i}$ to be such that $i=f_{a}^{-1}(a')$, $x_{i}$ exists
 and is unique from the injectivity of $f_{a}$.
Next, take $x_{j}$ to be such that $j=g_{i}(a)$.\\
So, it holds that $S(i,a)=(g_{i}(a),f_{a}(i))=(j,a')$.
Assume that $i=j$, then we  have  $S(i,a)=(g_{i}(a),f_{a}(i))=(i,a')$. Since $S$ is involutive, we have  also $S(i,a')=(g_{i}(a'),f_{a'}(i))=(i,a)$, that is $g_{i}(a)=g_{i}(a')=i$.
But this contradicts the injectivity of $g_{i}$, so $i \neq j$.
The uniqueness of each such relation is due to the fact that $i$ and $j$ are defined uniquely by the functions $f_{.}$ and $g_{.}$.
\end{proof}
\begin{claim}\label{M_right_cancel}
The monoid $M$  satisfies the condition  $(\tilde{C_{1}})$, that is   $M$ is right cancellative.
\end{claim}
The proof appears in the appendix.
\section{The right lcm of the generators is a Garside element}
The braid groups and the Artin groups of finite type are Garside groups which satisfy the condition that the right lcm of their set of atoms is a Garside element. In \cite{deh_Paris}, the authors considered this additional condition as a part of the definition of Garside groups and in \cite{deh_francais} it has been removed from the definition.
Indeed, in \cite{deh_francais} Dehornoy gives the following example of a monoid which is Garside and yet the right lcm of its atoms is not a Garside element: Let $M=\operatorname{Mon} \langle a,b \mid aba=b^{2} \rangle$, then $b^{3}$ represents a Garside element of $M$ and the right lcm of the atoms is $b^{2}$.
We prove that the structure group of a non-degenerate, braided and involutive solution is a Garside group in the sense of \cite{deh_Paris}, that is we prove the following result:
\begin{thm}\label{thm_delta_lcm_atoms}
Let $G$ be the structure group of a non-degenerate, braided and involutive solution $(X,S)$ and let  $M$ be the monoid with the same presentation.
 Then the right lcm of the atoms (i.e the elements of $X$) is a Garside element.
\end{thm}
In order to prove that, we show that the set of simple elements $\chi$, i.e the closure of $X$ under right complement and right lcm, is equal to the closure of $X$ under right  lcm (denoted by $\overline{X}^{\vee}$), where the empty word $\epsilon$ is added. So, this implies that $\Delta$, the right lcm of the simple elements, is the right lcm of  the elements in $X$.
We use the word reversing method developed by Dehornoy and the diagrams for word reversing. We illustrate in  example $1$ below the definition of the diagram and we refer the reader to \cite{deh_francais} and  \cite{deh_livre} for more details. The following proposition ensures in our case that reversing the word $u^{-1}v$ using the diagram amounts to computing a right lcm for the elements represented by $u$ and $v$.
\begin{prop}\cite[p.65]{deh_livre}
Assume that $M$ is a monoid associated with a coherent complement. Let $u,v$ be words on $X$ and let $g,g'$ respectively be their classes in $M$. If the word $u \setminus v$ exists and the elements $g$ and $g'$ admit a common right multiple, then the elements $g$ and $g'$  admit in $M$ a unique right lcm and $u\setminus v$ represents the element $g \setminus g'$.
\end{prop}
\textbf{Example 1 cont'}: Let us consider the monoid $M$ defined in example $1$ in Section $3$.\\
 (a) In order to reverse the word $x_{3}^{-1}x_{1}$, we begin with:$\begin{matrix}
    && \longrightarrow^{x_{1}} &    \\
 & \downarrow^{ x_{3}} &   \\
\end{matrix}$\\
The diagram corresponding to the reversing of the word  $x_{3}^{-1}x_{1}$, or in other words to the generators $x_{3}$ and $x_{1}$ is defined to be:
$\begin{matrix}
    &&
    \longrightarrow^{x_{1}} &    \\
 & \downarrow^{ x_{3}} &  & \downarrow^{x_{2}} \\
   &  & \longrightarrow ^{ x_{4}} &  &  \\
\end{matrix}$\\
Since it holds that $x_{1}x_{2}=x_{3}x_{4}$ in $M$,  $x_{1}\setminus x_{3}=x_{2}$ and $x_{3}\setminus x_{1}=x_{4}$.\\
(b) In order to reverse the word $x_{4}^{-2}x_{1}^{2}$ , we begin with:
$\begin{matrix}
     &&\longrightarrow^{x_{1}}  \longrightarrow^{x_{1}}   \\
 & \downarrow^{ x_{4}}  \\
 & \downarrow^{ x_{4}}  \\
\end{matrix}$\\
The diagram corresponding to the reversing of the word $x_{4}^{-2}x_{1}^{2}$ or to the words  $x_{4}^{2}$ and $x_{1}^{2}$ is defined to be:
$\begin{matrix}
    \longrightarrow^{x_{1}} && \longrightarrow^{x_{1}}   \\
 & \downarrow^{ x_{4}} &&\downarrow^{x_{3}}  && \downarrow^{x_{2}} \\
   &  & \longrightarrow ^{ x_{2}} &  &\longrightarrow ^{ x_{4}}  \\
  & \downarrow^{ x_{4}} &&\downarrow^{x_{1}}  && \downarrow^{x_{2}} \\
   &  & \longrightarrow ^{ x_{3}} &  &\longrightarrow ^{ x_{3}}  \\
\end{matrix}$\\
That is, it holds that $x_{1}^{2}x_{2}^{2}=x_{4}^{2}x_{3}^{2}$ in $M$  and since  $x_{1}^{2}=x_{2}^{2}$ and $x_{3}^{2}=x_{4}^{2}$, it holds that  $x_{1}^{4}=x_{2}^{4}=x_{3}^{4}=x_{4}^{4}$ in $M$.
So, a word representing the right lcm of $x_{1}^{2}$ and $x_{4}^{2}$ is the word $x_{1}^{4}$ or the word $x_{4}^{4}$...\\
We obtain from the diagram that the word $x_{2}^{2}$ represents the element $x_{1}^{2}\setminus x_{4}^{2}$ and the word $x_{3}^{2}$ represents the element  $x_{4}^{2}\setminus x_{1}^{2}$.\\
In order to prove that  $\chi=\overline{X}^{\vee}\bigcup \{\epsilon\}$, we need to show that every complement of simple elements is the right lcm of some generators. The following lemma from \cite{deh_francais} gives some rules of calculation on the complements.
\begin{lem}\cite[Lemma1.7]{deh_francais}\label{formules_complement_lcm}
Let $M$ be a monoid which satisfies the conditions $(C_{0})$, $(C_{1})$ and $(C_{2})$. Then for every $u,v,w \in M$, it holds that:\\
(i)  $u \setminus (v \vee w) = (u \setminus v) \vee (u \setminus w)$\\
(ii)  $(u  \vee v )\setminus w = (u \setminus v) \setminus (u \setminus w)$\\
(iii)  $u(v \vee w)=uv \vee uw$
\end{lem}
The following technical lemmas are the basis of induction for the proof of Theorem \ref{thm_delta_lcm_atoms}.
\begin{lem}\label{lem_MXin X}
It holds that  $M\setminus X \subseteq X \bigcup \{\epsilon\}$.
\end{lem}
\begin{proof}
It holds that $S(X \times X) \subseteq X \times X$, so $X\setminus X \subseteq X\bigcup \{\epsilon\}$ and this implies inductively that $M \setminus X \subseteq X \bigcup \{\epsilon\}$ (see the reversing diagram).
\end{proof}
\begin{lem} \label{lem_MlcmX_inlcmX}
It holds that  $M\setminus (\vee_{j=1}^{j=k} x_{i_{j}}) \subseteq \overline{X}^{\vee}\bigcup \{\epsilon\}$, where $x_{i_{j}} \in X$ for $ 1\leq j \leq k$.
\end{lem}
\begin{proof}
Let $u \in M$, then from lemma \ref{formules_complement_lcm} we have
inductively that $u \setminus (\vee_{j=1}^{j=k} x_{i_{j}}) = \vee_{j=1}^{j=k} (u \setminus x_{i_{j}})$.
From lemma \ref{lem_MXin X}, $u \setminus x_{i_{j}}$ belongs to $X$ so
$\vee_{j=1}^{j=k} (u \setminus x_{i_{j}})$ is in $\overline{X}^{\vee}\bigcup \{\epsilon\}$.
\end{proof}
Since the monoid $M$ is Garside, the set of simples  $\chi$  is finite and its  construction is  done in a finite number of steps in the following way:\\At the $0-$th step,  $\chi_{0}=X$.\\
At the first step,  $\chi_{1}=X \bigcup \{x_{i} \vee x_{j};$ for all  $x_{i},x_{j} \in X \}\bigcup
\{x_{i} \setminus x_{j};$ for all  $x_{i},x_{j} \in X \}$.\\
At the second step,  $\chi_{2}=\chi_{1} \bigcup \{u \vee v;$ for all  $u,v \in \chi_{1} \} \bigcup \{u \setminus v;$ for all  $u,v \in \chi_{1} \}$.\\
 We go on inductively and after a finite number of steps $k$, $\chi_{k}=\chi$.
\begin{prop}\label{prop_Simples}
It holds that $\chi=\overline{X}^{\vee}\bigcup \{\epsilon\}$.
\end{prop}
\begin{proof}
The proof is by induction on the number of steps $k$ in the construction of $\chi$. We show that each complement of simple elements is the right lcm of some generators.\\
At the first step, we have that $\{x_{i} \setminus x_{j};$ for all  $x_{i},x_{j} \in X \}= X\bigcup \{\epsilon\}$.\\
At the following steps, we do not consider the complements of the form $...\setminus x_{i}$ since these belong to $X$ (see lemma \ref{lem_MXin X}).\\
At the second step, the complements have the following form
$x_{i} \setminus (x_{l} \vee x_{m})$ or $(x_{i} \vee x_{j}) \setminus (x_{l} \vee x_{m})$ and these belong to  $\overline{X}^{\vee}\bigcup \{\epsilon\}$   from lemma \ref{lem_MlcmX_inlcmX}.\\
Assume that at the $k-$th step, all the complements obtained belong to
$\overline{X}^{\vee}\bigcup \{\epsilon\}$. So, from lemma \ref{lem_MlcmX_inlcmX},  all the elements of $\chi_{k}$ are  right lcm of generators. \\
At the $(k+1)-$th step, the complements have the following form $u \setminus v $, where $u,v \in \chi_{k}$. From the induction assumption,  $v$ is  a right lcm of generators and from lemma \ref{lem_MlcmX_inlcmX},  $u \setminus v $  belongs to  $\overline{X}^{\vee}\bigcup \{\epsilon\}$. \end{proof}
\textbf{Proof of Theorem \ref{thm_delta_lcm_atoms}}
\begin{proof}
The right lcm of the set of simples $\chi$ is a Garside element and since from Proposition \ref{prop_Simples}  $\chi
=\overline{X}^{\vee}\bigcup \{\epsilon\}$, we have that the right lcm of $\overline{X}^{\vee}$ is a Garside element.
From the uniqueness of the lcm, we have that the right lcm of $X$ is a Garside element.
\end{proof}
We show now that the length of a Garside element $\Delta$ is $n$. In order to show that, we prove that the right lcm of $k$ different generators has length $k$ using the following technical lemmas.
\begin{lem}\label{lem:h3equalx}
Let $h_{i}, x$ be all different elements in $X$, for $1 \leq i \leq 3$.
If $(h_{1} \setminus h_{2}) \setminus (h_{1} \setminus h_{3})=  (h_{1} \setminus h_{2}) \setminus (h_{1} \setminus x)$,  then $h_{3}=x$.
\end{lem}
\begin{proof}
It holds that each expression $h_{i} \setminus h_{j}$ belongs to $X$, so let denote $h_{1} \setminus h_{2}$ by $a$, $h_{1} \setminus h_{3}$ by $b$ and $h_{1} \setminus x$ by $c$.
Then we have $a\setminus b=a\setminus c$, where $a,b,c \in X$.
From lemma \ref{form_compl}, this means that $g^{-1}_{a}(b)=g^{-1}_{a}(c)$. But this contradicts the fact that $(X,S)$ is non-degenerate.
\end{proof}
\begin{lem}\label{lem_compl_egal}
Let $h_{i}, x$ be all different elements in $X$, for $1 \leq i \leq k$.
Then $(\vee_{i=1}^{i=k}h_{i}) \setminus x$ is not equal to the empty word.
\end{lem}
\begin{proof}
If $k=3$, we obtain from lemma \ref{formules_complement_lcm}(ii)  that\\
$(\vee_{i=1}^{i=3}h_{i}) \setminus x = ((h_{1} \setminus h_{2}) \setminus (h_{1} \setminus h_{3})) \setminus ((h_{1} \setminus h_{2}) \setminus (h_{1} \setminus x))$.\\
If $(\vee_{i=1}^{i=3}h_{i}) \setminus x =\epsilon$, then since all the expressions of the form $(h_{i} \setminus h_{j}) \setminus (h_{l} \setminus h_{m})$ belong to $X$, we have that \\
$(h_{1} \setminus h_{2}) \setminus (h_{1} \setminus h_{3})=  (h_{1} \setminus h_{2}) \setminus (h_{1} \setminus x)$.\\
From lemma \ref{lem:h3equalx}, this implies that  $h_{3}=x$.
But this is a contradiction.\\
If $k=4$, we obtain from lemma \ref{formules_complement_lcm}(ii) that\\
$(\vee_{i=1}^{i=4}h_{i}) \setminus x = (((h_{1} \setminus h_{2}) \setminus (h_{1} \setminus h_{3})) \setminus ((h_{1} \setminus h_{2}) \setminus (h_{1} \setminus h_{4}))) \setminus (((h_{1} \setminus h_{2}) \setminus (h_{1} \setminus h_{3})) \setminus ((h_{1} \setminus h_{2}) \setminus (h_{1} \setminus x)))$.\\
If $(\vee_{i=1}^{i=4}h_{i}) \setminus x =\epsilon$, then since all the expressions of the form $(h_{i} \setminus h_{j}) \setminus (h_{l} \setminus h_{m})$ belong to $X$, we have that \\
$((h_{1} \setminus h_{2}) \setminus (h_{1} \setminus h_{3})) \setminus ((h_{1} \setminus h_{2}) \setminus (h_{1} \setminus h_{4}))=((h_{1} \setminus h_{2}) \setminus (h_{1} \setminus h_{3})) \setminus ((h_{1} \setminus h_{2}) \setminus (h_{1} \setminus x))$\\ and this implies that \\
$(h_{1} \setminus h_{2}) \setminus (h_{1} \setminus h_{4})=(h_{1} \setminus h_{2}) \setminus (h_{1} \setminus x)$\\
and from lemma \ref{lem:h3equalx} we obtain $x=h_{4}$ and  this is a contradiction.\\
For a general $k$, the formula is even more complicated but exactly the same argument we used in the cases $k=3$ and $k=4$  holds.
\end{proof}
\begin{thm}\label{thm_Garside_length_n}
Let $G$ be the structure group of a non-degenerate, braided and involutive solution $(X,S)$, where $X=\{x_{1},..,x_{n}\}$ and let  $M$ be the monoid with the same presentation.
Let $\Delta$ be a Garside element in $M$.
Then the length of $\Delta$ is  $n$.
\end{thm}
\begin{proof}
From theorem \ref{thm_delta_lcm_atoms}, $\Delta$ represents the right lcm of the elements in $X$, that is
$\Delta= x_{1} \vee x_{2} \vee ...\vee x_{n}$ in $M$.
We show by induction that a word representing the right lcm  $x_{i_{1}}\vee x_{i_{2}} \vee..\vee x_{i_{k}}$ has length  $k$, where $x_{i_{j}} \neq x_{i_{l}}$ for $j \neq l$.\\
If $k=2$, then there are different generators $a,b$ such that $S(x_{i_{1}},a)=(x_{i_{2}},b)$, so $x_{i_{1}}a=x_{i_{2}}b$
is a relation in $M$ and the right lcm of $x_{i_{1}},x_{i_{2}}$ has length $2$.
Assume that the  right lcm  $x_{i_{1}}\vee x_{i_{2}} \vee..\vee x_{i_{k-1}}$ has length  $k-1$. Then the right lcm  $x_{i_{1}}\vee x_{i_{2}} \vee..\vee x_{i_{k-1}} \vee x_{i_{k}}$ is obtained from the reversing diagram corresponding to the  words $x_{i_{1}}\vee x_{i_{2}} \vee..\vee x_{i_{k-1}}$ and $x_{i_{k}}$.
From lemma \ref{lem_compl_egal}, $(x_{i_{1}}\vee x_{i_{2}} \vee..\vee x_{i_{k-1}})\setminus x_{i_{k}}$ is not equal to the empty word, so from lemma \ref{lem_MXin X} it has length $1$.
So, the right lcm  $x_{i_{1}}\vee x_{i_{2}} \vee..\vee x_{i_{k}}$ has length  $k$ and this implies that  $x_{1}\vee x_{2} \vee..\vee x_{n}$ has length  $n$.
\end{proof}
\section{The homological dimension of  structure groups is bounded}
 In \cite{deh_homologie}, the authors construct a  resolution of $\Bbb
Z$ (as trivial ${\Bbb Z}M$-module) by free ${\Bbb Z}M$-modules,
when $M$ satisfies some conditions. They show that if $M$ is a Garside monoid then the resolution
defined in \cite{charney}, using another approach,  is isomorphic to the resolution they define.
We will use the resolution from \cite{deh_homologie} in
order to show that  the homological dimension of the structure
group corresponding to a ``set-theoretical'' solution $(X,S)$ of the
quantum Yang-Baxter equation is bounded from above by the number of
generators in $X$.
\begin{defn}\cite{deh_homologie}
A monoid $M$ is \textbf{left Noetherian}  if left divisibility is
well-founded, i.e there is no infinite descending sequence of left
divisions.  A monoid $M$ is \textbf{a locally Gaussian monoid} if
$M$ is cancellative, left and right Noetherian and every two
elements of $M$ admitting a common right (resp.left) multiple
admits a right (resp.left) lcm. A monoid $M$ is \textbf{ (locally) Garside} if it is  (locally) Gaussian and it admits a finite generating subset $Y$ that is closed under right and left lcm and under left and right complements.
\end{defn}
 In \cite{deh_homologie}, the condition for  the resolution
 of $\Bbb Z$  by free ${\Bbb Z}M$-modules is that $M$ is a
 locally Gaussian monoid, so we can use it for the monoid $M$ with the same
 presentation as the structure group corresponding to
a ``set-theoretical'' solution of the quantum Yang-Baxter equation,
since these are Garside.
\begin{prop}\cite[Prop.2.9-2.10]{deh_homologie}
If $M$ is a (locally) Garside monoid, then the resolution is
finite, so $M$ is of type FL. If $M$ is Garside and $G$ is its
group of fractions, then the functor ${\Bbb Z}G\otimes_{{\Bbb
Z}M}-$ is exact, so it follows that every Garside group is of type
FL.
\end{prop}
We will describe in a few words the resolution constructed in
\cite{deh_homologie} and we refer the reader for more details.
 Let  $M$ be a locally Gaussian monoid and let $\chi$ be a
 generating subset of $M$, not containing the empty word $\epsilon$, that is closed
 under left and right lcm such that $\chi \bigcup \{\epsilon\}$ is closed under left and right
 complement. Let $<$ be a linear ordering on $\chi$ such that $\alpha < \beta$
holds whenever $\beta$ is a proper right divisor of $\alpha$: this
is possible since right division in $M$ has no cycle.
\begin{defn}\cite{deh_homologie}
For $n\geq 0$, $\chi^{[n]}$ denotes the family of all strictly
increasing $n$-tuples $(\alpha_{1},\alpha_{2},..,\alpha_{n})$ in
$\chi$ such that $\alpha_{1},\alpha_{2},..,\alpha_{n}$ admit a
left lcm. The free ${\Bbb Z}M$-module generated by $\chi^{[n]}$ is
denoted by $C_{n}$ and the generator associated with an element
$A$ of $\chi^{[n]}$ is denoted $[A]$ and it is called an
\textbf{$n-$cell}. The unique $0-$cell is denoted by
$[\emptyset]$. An $n-$cell $[\alpha_{1},\alpha_{2},..,\alpha_{n}]$
is \textbf{descending} if $\alpha_{i+1}$ is a proper right divisor
of $\alpha_{i}$ for each $i$. The submodule of $C_{n}$ generated
by descending  $n-$cells is denoted by $C'_{n}$.
\end{defn}
The boundary maps $\partial_{n}:C_{n} \rightarrow C_{n-1}$ are
explicitly given and it holds that the boundary of a descending
cell consists of descending cells exclusively, that is
$\partial_{n}(C'_{n})\subseteq C'_{n-1}$. The restriction of
$\partial_{n}$ to $C'_{n}$ is denoted by $\partial'_{n}$.
 \begin{prop}\cite[Prop.3.2]{deh_homologie}
 For each locally Gaussian monoid $M$, the subcomplex
 $(C'_{*}, \partial'_{*})$ of $(C_{*},\partial_{*})$ is a finite
 resolution of the trivial ${\Bbb Z}M$-module $\Bbb Z$ by free ${\Bbb
 Z}M$-modules.
 \end{prop}
\begin{cor}\cite[Cor.3.6]{deh_homologie}\label{limiter_dimension}
Assume that $M$ is a  locally Gaussian monoid admitting a
generating set $\chi$ such that $\chi \bigcup \{\epsilon\}$ is closed under left and right complement and
lcm and such that the norm of every element in $\chi$ is bounded
above by $n$. Then the (co)homological dimension of $M$ is at most
$n$.
\end{cor}
Using Corollary \ref{limiter_dimension}, we prove the following result:
\begin{thm}
Let $(X,S)$ be  a ``set-theoretical'' solution  of the quantum
Yang-Baxter equation, where $X=\{x_{1},..,x_{n}\}$ and $(X,S)$ is\\
non-degenerate, braided and involutive. Let $G$ be the structure
group corresponding to $(X,S)$.
 Then  the (co)homological dimension of $G$ is bounded from above by $n$, the number of generators in $X$.
\end{thm}
\begin{proof}
The set of simples $\chi$ satisfies the conditions of Corollary \ref{limiter_dimension} and the norm of every element in $\chi$ is bounded by $n$, since this is the length of the right lcm of $\chi$ (from Theorems \ref{thm_delta_lcm_atoms} and \ref{thm_Garside_length_n}).
So, the (co)homological dimension of $G$ is bounded from above by $n$.
\end{proof}
\section{The structure group of $(X,S)$ is Garside $\Delta-$pure iff $(X,S)$ is indecomposable}
We refer the reader to sections $2$ and $3$ for the definitions of indecomposable solutions and Garside $\Delta-$pure monoids respectively.
In \cite{etingof}, the authors give a classification of non-degenerate, braided and involutive solutions with $X$ up to $8$ elements, considering their decomposability and other properties.
In \cite{rump}, Rump proves Gateva-Ivanova's conjecture (and also the authors'of \cite{etingof}) that every square-free, non-degenerate, involutive and braided solution is decomposable. Moreover, he constructs an indecomposable solution with $X$ infinite which shows that an extension to infinite $X$ is false.
We find a criteria for decomposability of the solution involving the Garside structure of the structure group (monoid), that is we prove the following result:
\begin{thm}\label{thm_deltapure_indecomp}
Let $G$ be the structure group corresponding to a non-degenerate, braided and involutive solution $(X,S)$ and let  $M$ be the monoid with the same presentation.
Then $M$ is   Garside $\Delta-$pure if and only if $(X,S)$ is indecomposable.
\end{thm}
\textbf{Example 2:}
Let $X=\{x_{1},x_{2},x_{3}\}$ and let $S(x_{i},x_{j})=(f(j),f^{-1}(i))$, where $f=(1,2,3)$, be a permutation solution. This is an indecomposable solution since $f$ is cyclic (see \cite{etingof}).
The defining relations in $M$ are: $x_{1}^{2}=x_{2}x_{3}$, $x_{2}^{2}=x_{3}x_{1}$ and $x_{3}^{2}=x_{1}x_{2}$.
So, $X \setminus x_{1}=\{x_{3}\}$, $X \setminus x_{2}=\{x_{1}\}$ and $X \setminus x_{3}=\{x_{2}\}$.
Using the reversing diagram, we obtain inductively that  $M \setminus x_{i}=X\bigcap \{\epsilon\}$ for $1 \leq i \leq 3$, that is $M$ is Garside $\Delta-$pure, since $\Delta_{1}=\Delta_{2}=\Delta_{3}$.
As an example, $x_{2}\setminus x_{1}=x_{3}$, so $x_{2}x_{1}\setminus x_{1}=x_{2}$ and so  $x_{2}x_{1}x_{3}\setminus x_{1}=x_{1}$ that is  $X \bigcap \{\epsilon\} \subseteq M \setminus x_{1}$ and since  $M \setminus x_{1}\subseteq X\bigcap \{\epsilon\}$ (see lemma \ref{lem_MXin X}) we have the equality.
\begin{lem}\label{lem_invariantYY}
Let $Y$ be an invariant subset of $X$. Then $Y\setminus Y \subseteq Y$.
\end{lem}
\begin{proof}
If $Y$ is an invariant subset of $X$, then $S(Y \times Y) \subseteq Y \times Y$, so $Y\setminus Y \subseteq Y$.
\end{proof}
\begin{lem}\label{lem_YYimpliqMY}
Let $(X,S)$ be the union of non-degenerate invariant subsets  $Y$ and $Z$. Then $M\setminus Y \subseteq Y$ and $M\setminus Z \subseteq Z$.
\end{lem}
\begin{proof}
If $Y$ and $Z$ are invariant subsets of $X$, then $Y\setminus Y \subseteq Y$ and $Z\setminus Z \subseteq Z$, from lemma \ref{lem_invariantYY}.
From Proposition \ref{prop_eting_YZ}, it holds that
$S(Y \times Z) \subseteq Z \times Y$ and $S(Z \times Y) \subseteq Y \times Z$, so
$Z\setminus Y \subseteq Y$ and $Y\setminus Z \subseteq Z$.
That is, we have that  $Y\setminus Y \subseteq Y$ and $Z\setminus Y \subseteq Y$, so $X\setminus Y \subseteq Y$ and this implies inductively that $M\setminus Y \subseteq Y$ (see the reversing diagram). The same holds for $Z$.
\end{proof}
\begin{rem}\label{rem_interpret_compl_g}
When $w\setminus x$ is not equal to the empty word, then we can interpret $w\setminus x$ in terms of the functions $g^{-1}_{*}$ using the reversing diagram corresponding to  the words $w=h_{1}h_{2}..h_{k}$ and $x$, where $h_{i},x \in X$ and for brevity of notation we write $g^{-1}_{i}(x)$ for $g^{-1}_{h_{i}}(x)$:\\

$\begin{matrix}
    &&
    \longrightarrow^{h_{1}} && \longrightarrow^{h_{2}}   &..& \longrightarrow^{h_{k}}\\
 & \downarrow^{ x} &&\downarrow^{g^{-1}_{1}(x)}  && \downarrow^{g^{-1}_{2}g^{-1}_{1}(x)} & ..&\downarrow^{g^{-1}_{k}..g^{-1}_{2}g^{-1}_{1}(x)}\\
   &  & \longrightarrow ^{ g^{-1}_{x}(h_{1})} &  &\longrightarrow ^{ g^{-1}_{*}(.)} &..& \longrightarrow^{ g^{-1}_{**}(..)}\\
\end{matrix}$\\
That is, $h_{1}h_{2}..h_{k}\setminus x = g^{-1}_{k}..g^{-1}_{2}g^{-1}_{1}(x)$ and this is equal to $g^{-1}_{w}(x)$, since the action on $X$ is a right action.
Having a glance at the reversing diagram, we can note  that if  $w\setminus x$ is not equal to the empty word, then none of the expressions $g^{-1}_{1}(x)$, $g^{-1}_{2}g^{-1}_{1}(x)$,..,  $g^{-1}_{k}..g^{-1}_{2}g^{-1}_{1}(x)$ can be equal to the empty word.
\end{rem}
 Using exactly the same arguments as in the proof of Proposition \ref{pro_etin_Gtransitiv} in \cite{etingof},  we prove  in the following lemma that if $(X,S)$ is indecomposable then the action of $M$ on $X$ is transitive.
\begin{lem}\label{transitive_M}
Let  $(X,S)$ be a non-degenerate, involutive and braided solution.
 If $(X,S)$ is indecomposable then the action of $M$ on $X$ is transitive.
\end{lem}
\begin{proof}
Assume that the action $x \rightarrow g^{-1}_{x}$ of $M$ on $X$  is not transitive.
Then we can construct two  nonempty subsets $Y$ and $Z$ of $X$ such that $Y\bigcap Z =\emptyset$,
$X= Y \bigcup Z$ and $Y$ and $Z$ are $M-$invariant.
It remains to  show that $Y$ and $Z$ are invariant subsets of $X$ and that these are non-degenerate.
These subsets are invariant under $g_{x}$  and hence under $T^{-1}$ and $T$,  since $T^{-1}(x)=g^{-1}_{x}(x)$ (see Proposition \ref{Prop_def_T}).
Since $f_{x}=Tg^{-1}_{x}T^{-1}$ from Proposition \ref{Prop_def_T}, $Y$ and $Z$ are also invariant under $f_{x}$.
So,   $Y$ and $Z$ are invariant subsets of $X$.
These subsets are non-degenerate since the functions $f_{y},g_{y}$ for $y\in Y$ and $f_{z},g_{z}$ for $z\in Z$ are bijective ($(X,S)$ is non-degenerate).
\end{proof}
\textbf{Proof of Theorem \ref{thm_deltapure_indecomp}}
\begin{proof}
Assume that $(X,S)$ is decomposable, that is $(X,S)$ is the union of non-degenerate invariant subsets  $Y$ and $Z$.
From lemma \ref{lem_YYimpliqMY}, we have $M\setminus Y \subseteq Y$ and $M\setminus Z \subseteq Z$.
Let $y\in Y$ and $z\in Z$, then $\Delta_{y}= \vee (M\setminus y)$  cannot be the same as $\Delta_{z}= \vee (M\setminus z)$. So, $M$ is not $\Delta-$pure.\\
Assume that $(X,S)$ is indecomposable and we need to show that $M$ is Garside $\Delta-$pure. We prove for every $x \in X$ that $M\setminus x=X$ and this implies that $\Delta_{x}= \vee (M\setminus x)=\Delta$. Clearly $M\setminus x \subseteq X$, so it remains to show that  for every $y \in X$ there is an element  $w$ in $M$ such that $w\setminus x=y$.
Since $(X,S)$ is indecomposable, from Proposition \ref{braided_actions} and lemma \ref{transitive_M}, we have that the assignment $x \rightarrow g^{-1}_{x}$ is a transitive right action of $M$ on $X$, that is for every $x,y \in X$, there is an element $w \in M$ such that $g^{-1}_{w}(x)=y$.
That means that   for every $y \in X$ there is an element  $w$ in $M$ such that $w\setminus x=y$.
\end{proof}
\section{The converse implication: from 'tableau' Garside monoids to structure groups}
We prove the following:
\begin{thm}\label{garside_struct_group}
Let $M=\operatorname{Mon} \langle X\mid R \rangle$ be a  Garside monoid such that:
1) $X=\{x_{1},..,x_{n}\}$.\\
2) There are $n(n-1)/2$ relations in $R$.\\
3) Each side of a relation in $R$ has length 2.\\
4) If the  word $x_{i}x_{j}$ appears in $R$, then it appears only once.\\
 Then there exists a function $S: X \times X \rightarrow X \times X$ such that $(X,S)$ is  a non-degenerate, involutive and braided ``set-theoretical'' solution and $G=\operatorname{Gp} \langle X\mid R \rangle$ is its structure group.\\
 If in  addition: 5) There is no  word $x_{i}^{2}$  in $R$.\\
Then  $(X,S)$ is also square-free.
 \end{thm}
In order to prove Theorem \ref{garside_struct_group}, we need to introduce the concept of left coherence and some new terminology.
The following definitions are taken from \cite{deh_francais} and \cite{deh_homologie}, but we do not use exactly the same notations.
\begin{defn}
Let $M$ be a monoid and let $x,y$ be elements in $M$. We say that $x$ is \textbf{a right divisor} of $y$ if there is an element $t$ in $M$
such that $y=tx$ and  $x$ is \textbf{a proper} right divisor of $y$ if
in addition $t \neq 1$.
\end{defn}
\begin{defn}
Let $M$ be a monoid and let $x,y$ be elements in $M$. Say that $z$
is
\textbf{a left least common multiple or left lcm} of $x$ and $y$ if:\\
  $1.$ $x$ and $y$ are right divisors of $z$ \\
  $2.$ if there is an element $w$ in $M$ such that $x$ and $y$ are right divisors of $w$, then
   $z$ is  right divisor of $w$.
   The notation is $z=x\widetilde{\vee} y$.\\
  \textbf{The complement at left of
   $y$ on $x$} is defined to be an element $c\in M$ such that $x\widetilde{\vee} y= cx$ and the notation is
   $c=y \widetilde{\setminus} x$ and $x\widetilde{\vee} y=(x \widetilde{\setminus} y)y$.
\end{defn}
\begin{defn}
Let $M$ be a monoid. The \textbf{left coherence} on $M$ is satisfied if it holds for any $x,y,z \in M$:\\
$((x\widetilde{\setminus} y)\widetilde{\setminus} (z\widetilde{\setminus} y))\widetilde{\setminus} ((x\widetilde{\setminus} z)\widetilde{\setminus} (y\widetilde{\setminus} z))\equiv^{+} \epsilon$
It is also called the \textbf{left cube condition}.
\end{defn}
We show that if $(X,S)$ is a non-degenerate and  involutive ``set-theoretical'' solution, then  $(X,S)$ is braided if and only if  $X$ is coherent and  left coherent.
 The left coherence on $X$  is satisfied if the following condition on all  $x_{i},x_{j},x_{k}$ in $X$ is satisfied:\\
$(x_{i}\widetilde{\setminus} x_{j})\widetilde{\setminus} (x_{k}\widetilde{\setminus} x_{j})= (x_{i}\widetilde{\setminus} x_{k})\widetilde{\setminus} (x_{j}\widetilde{\setminus} x_{k})$,\\
where the equality is in the free monoid since  the complement on the left is totally defined and  its range is $X$ (from lemma \ref{lem_cancel}). Note that as  in the proof of the coherence,  left coherence on $X$ implies left coherence on $M$, since the monoid $M$ is atomic.\\
Clearly, the following implication is derived from Theorem \ref{theo:garside}:
\begin{lem}\label{braided_implies_leftcoherent}
Assume  $(X,S)$ is non-degenerate and  involutive.\\ If  $(X,S)$ is braided, then $X$ is coherent and left coherent.
\end{lem}
\begin{proof}
Assume $(X,S)$ is braided, then from Theorem \ref{theo:garside}, the structure group corresponding to $(X,S)$ is Garside since  the monoid with the same presentation is Garside. So, $X$ is coherent and left coherent (see for example  \cite[Prop.1.4]{deh_homologie} and \cite[Lemma1.7]{deh_francais}).
 \end{proof}
The proof of the converse implication is less trivial and requires a lot of calculations. Before we proceed, we need first express the left complement in terms of the functions $f_{*}$, some of the calculations are symmetric to those done in Section 4.2 with the right complement.
\begin{lem}\label{compl_gauche}
 Assume $(X,S)$ is non-degenerate.\\
 Let $x_{i},x_{j}$ be different elements in $X$.\\
Then $x_{j}\widetilde{\setminus} x_{i}=f_{i}^{-1}(j)$.\\
Note that $x_{j}\widetilde{\setminus} x_{j}=\epsilon$.
\end{lem}
\begin{proof}
If $S(a,i)=(b,j)$, then $x_{j}\widetilde{\setminus} x_{i}=a$.
But by definition of $S$, we have that $S(a,i)=(g_{a}(i),f_{i}(a))$,
 so it holds that $f_{i}(a)=j$
 which gives $a=f_{i}^{-1}(j)$.
\end{proof}
\begin{lem}\label{formule_gauche}
Assume $(X,S)$ is non-degenerate and  involutive.\\ Let $x_{i}, x_{k}$ be elements in $X$.\\
Then $f^{-1}_{k}(i)=g_{f^{-1}_{i}(k)}(i)$.
\end{lem}
\begin{proof}
$S$ is involutive so we have from claim \ref{debut_form} that for every $x_{i},x_{j}\in X$, $f_{f_{i}(j)}g_{j}(i)=i$.\\
Let replace in this formula  $j$  by $f^{-1}_{i}(k)$, then we obtain \\
$i= f_{f_{i}(f^{-1}_{i}(k))}g_{f^{-1}_{i}(k)}(i)=f_{k}g_{f^{-1}_{i}(k)}(i)$\\
So, we have $f^{-1}_{k}(i)=g_{f^{-1}_{i}(k)}(i)$.
\end{proof}
\begin{lem}\label{cor_Xcoherent_equations}
Assume $(X,S)$ is non-degenerate.\\
If  $X$ is coherent and left coherent,  then for every $i,j,k$ the following equations hold:\\
(A) $f_{j}f_{f^{-1}_{j}(k)}=f_{k}f_{f_{k}^{-1}(j)}$\\
(B) $ g_{i}g_{g_{i}^{-1}(k)}=g_{k}g_{g_{k}^{-1}(i)}$
\end{lem}
\begin{proof}
From lemma \ref{compl_gauche}, we have the following formulas:\\
 $(x_{i}\widetilde{\setminus} x_{j})\widetilde{\setminus} (x_{k}\widetilde{\setminus} x_{j})=
  f^{-1}_{f^{-1}_{j}(k)} f_{j}^{-1} (i)$ for all different $1 \leq i,j,k \leq n$\\
$ (x_{i}\widetilde{\setminus} x_{k})\widetilde{\setminus} (x_{j}\widetilde{\setminus} x_{k})=
  f^{-1}_{f_{k}^{-1}(j)} f_{k}^{-1} (i)$ for all different $1 \leq i,j,k \leq n$ \\
So, if $X$ is left coherent then for all different $1 \leq i,j,k \leq n$ it holds that:
 (*) $f^{-1}_{f^{-1}_{j}(k)} f_{j}^{-1} (i)=$ $f^{-1}_{f_{k}^{-1}(j)} f_{k}^{-1} (i)$\\
 If $j=k$, the  equality (*) holds trivially, so let fix $j$ and $k$ such that $j \neq k$.
 We have to show that (*) holds also when $i=j$ or $i=k$.\\
 The functions $F_{1}=$ $f^{-1}_{f^{-1}_{j}(k)} f_{j}^{-1}$ and $F_{2}=$ $f^{-1}_{f_{k}^{-1}(j)} f_{k}^{-1}$ are bijective, since these are compositions of bijective functions. Moreover,  $F_{1}$ and $F_{2}$ are equal in
 $\{1,2,..,n\} \setminus \{j,k\}$. So, there are two possibilities:\\
 Case (1): $F_{1}(k)=F_{2}(k)$ and $F_{1}(j)=F_{2}(j)$ or \\ Case (2): $F_{1}(k)=F_{2}(j)$ and $F_{1}(j)=F_{2}(k)$.\\
 Assume by contradiction that Case (2) occurs, so there is $1 \leq m \leq n$ such that
  $m$= $f^{-1}_{f^{-1}_{j}(k)} f_{j}^{-1} (k)=$ $f^{-1}_{f_{k}^{-1}(j)} f_{k}^{-1} (j)$\\
  that is $f_{f^{-1}_{j}(k)}(m)= f_{j}^{-1} (k)$ and $f_{f_{k}^{-1}(j)} (m)=f_{k}^{-1} (j)$\\
 That is $S(m,f^{-1}_{j}(k))=(m,f^{-1}_{j}(k))$ and $S(m,f^{-1}_{k}(j))=(m,f^{-1}_{k}(j))$,
 since $(X,S)$ is involutive.
 So, $g_{m}(f^{-1}_{j}(k))=m$ and $g_{m}(f^{-1}_{k}(j))=m$. Since $g_{m}$ is bijective, this implies that
 there is $1 \leq l \leq n$ such that $l=$ $f^{-1}_{j}(k)=f^{-1}_{k}(j)$, that is $S(l,j)=(l,k)$.
 But, since $j\neq k$,  this contradicts the fact that $(X,S)$ is involutive (see remark \ref{rem_cancel}).
 So, Case (2) cannot occur, that is for all $1 \leq i,j,k \leq n$ the equality (*) holds.
  Since the functions $f_{.}$ are bijective, (*) is equivalent to (A).\\
 Equation (B) is obtained in the same way using the coherence of $X$ (see lemma \ref{form_compl}).
 \end{proof}
\begin{prop}\label{coherence_implies_braided}
Let $(X,S)$ be  a non-degenerate and  involutive ``set-theoretical'' solution.\\ If  $X$ is coherent and left coherent, then $(X,S)$ is braided.
\end{prop}
\begin{proof}
We need to show that the functions $f_{.}$ and $g_{.}$ satisfy the following equations from lemma \ref{debut_form}:\\
(1)  $f_{j}f_{i}=f_{f_{j}(i)}f_{g_{i}(j)}$, $1 \leq i,j \leq n$. \\
(2) $g_{i}g_{j}=g_{g_{i}(j)}g_{f_{j}(i)}$, $1 \leq i,j \leq n$. \\
(3) $f_{g_{f_{l}(m)}(j)}g_{m}(l)=g_{f_{g_{l}(j)}(m)}f_{j}(l)$, $1 \leq j,l,m \leq n$.\\
From lemma \ref{cor_Xcoherent_equations}, we have for $1 \leq j, k \leq n$ that:\\
 (A)  $f_{j}f_{f^{-1}_{j}(k)}=f_{k}f_{f^{-1}_{k}(j)}$.\\
Assume $m=f^{-1}_{j}(k)$, that is $k=f_{j}(m)$ and let replace in formula (A) $f^{-1}_{j}(k)$ by $m$  and $k$ by $f_{j}(m)$, then we obtain:\\
$f_{j}f_{m}=f_{f_{j}(m)}f_{f^{-1}_{f_{j}(m)}(j)}$.\\
In order to show that (1) holds, we need to show that $f^{-1}_{f_{j}(m)}(j)=g_{m}(j)$.\\
From lemma \ref{formule_gauche}, we have $f^{-1}_{l}(j)=g_{f^{-1}_{j}(l)}(j)$ for every $j,l$, so by replacing $l$ by  $f_{j}(m)$, we obtain $f^{-1}_{f_{j}(m)}(j)=g_{m}(j)$. So, (1) holds.\\
From Corollary \ref{cor_Xcoherent_equations}, we have for $1 \leq j \neq k \leq n$ that:\\ (B) $ g_{i}g_{g_{i}^{-1}(k)}=g_{k}g_{g_{k}^{-1}(i)}$.\\
Assume $m=g^{-1}_{i}(k)$, that is $k=g_{i}(m)$ and let replace in formula (B) $g^{-1}_{i}(k)$ by $m$  and $k$ by $g_{i}(m)$, then we obtain:\\
$g_{i}g_{m}=g_{g_{i}(m)}g_{g^{-1}_{g_{i}(m)}(i)}$.\\
In order to show that (2) holds, we need to show that $g^{-1}_{g_{i}(m)}(i)=f_{m}(i)$.\\
From lemma \ref{formule}, we have $g^{-1}_{l}(i)=f_{g^{-1}_{i}(l)}(i)$, so by replacing $l$ by  $g_{i}(m)$, we obtain $g^{-1}_{g_{i}(m)}(i)=f_{m}(i)$. So, (2) holds.\\
It remains to show that (3) holds.
From (1), we have for $1 \leq i,j \leq n$ that  $f_{j}f_{i}=f_{f_{j}(i)}f_{g_{i}(j)}$ and this is equivalent to
$f_{g_{i}(j)}=f^{-1}_{f_{j}(i)}f_{j}f_{i}$.\\
Let replace $i$ by $f_{l}(m)$ for some $1 \leq l,m \leq n$ in the formula.\\
We obtain $f_{g_{f_{l}(m)}(j)}=f^{-1}_{f_{j}f_{l}(m)}f_{j}f_{f_{l}(m)}$.\\
By applying these functions on  $g_{m}(l)$ on both sides, we obtain\\
$f_{g_{f_{l}(m)}(j)}g_{m}(l)=f^{-1}_{f_{j}f_{l}(m)}f_{j}f_{f_{l}(m)}g_{m}(l)$.\\
Since $(X,S)$ is involutive, we have $f_{f_{l}(m)}g_{m}(l)=l$ (see lemma \ref{debut_form}).\\ So,
$f_{g_{f_{l}(m)}(j)}g_{m}(l)=f^{-1}_{f_{j}f_{l}(m)}f_{j}(l)$\\
From lemma \ref{formule_gauche}, we have $f^{-1}_{i}(k)=g_{f^{-1}_{k}(i)}(k)$ for every $i,k$, so  replacing $i$ by $f_{j}f_{l}(m)$ and $k$ by $f_{j}(l)$ gives
$f^{-1}_{f_{j}f_{l}(m)}(f_{j}(l))=g_{f^{-1}_{f_{j}(l)}f_{j}f_{l}(m)}(f_{j}(l))$.\\
So, $f_{g_{f_{l}(m)}(j)}g_{m}(l)=g_{f^{-1}_{f_{j}(l)}f_{j}f_{l}(m)}(f_{j}(l))$.\\
From (1), we have that $f^{-1}_{f_{j}(l)}f_{j}f_{l}(m)=f_{g_{l}(j)}(m)$,\\
 so $f_{g_{f_{l}(m)}(j)}g_{m}(l)=g_{f_{g_{l}(j)}(m)}f_{j}(l)$, that is (3) holds.
\end{proof}
\textbf{Proof of Theorem \ref{garside_struct_group}}
\begin{proof}
Let $M=\operatorname{Mon} \langle X\mid R \rangle$ be a  Garside monoid such that:\\
1) $X=\{x_{1},..,x_{n}\}$.\\
2) There are $n(n-1)/2$ relations in $R$.\\
3) Each side of a relation in $R$ has length 2.\\
4)  If the  word $x_{i}x_{j}$ appears in $R$, then it appears only once.\\
First,  we define a function $S: X \times X \rightarrow X \times X$ and $2n$ functions $f_{i},g_{i}$ for $1 \leq i \leq n$, such that $S(i,j)=(g_{i}(j),f_{j}(i))$ in the following way: If there is a relation $x_{i}x_{j}=x_{k}x_{l}$ then we define $S(i,j)=(k,l)$,  $S(k,l)=(i,j)$ and we define $g_{i}(j)=k$, $f_{j}(i)=l$, $g_{k}(l)=i$ and $f_{l}(k)=j$.
If the word $x_{i}x_{j}$ does not appear as a side of a relation, then we define $S(i,j)=(i,j)$ and we define $g_{i}(j)=i$ and $f_{j}(i)=j$.\\
We show that the functions $f_{i}$ and $g_{i}$ are well defined for $1 \leq i \leq n$:
Assume $g_{i}(j)=k$ and $g_{i}(j)=k'$ for some $1 \leq j,k,k' \leq n$ and $k \neq k'$, then it means from the definition of $S$ that
$S(i,j)=(k,.)$ and $S(i,j)=(k',..)$ that is the word $x_{i}x_{j}$ appears twice in $R$ and this contradicts (4).
The same argument holds for the proof that the functions $f_{i}$  are well defined.\\
We show that the functions $f_{i}$ and $g_{i}$ are bijective for $1 \leq i \leq n$:
Assume $g_{i}(j)=k$ and $g_{i}(j')=k$ for some $1 \leq j,j',k \leq n$ and $j \neq j'$, then  from the definition of $S$ we have $S(i,j)=(k,l)$ and $S(i,j')=(k,l')$ for some $1 \leq l,l'\leq n$,  that is there are the following  two defining relations in $R$: $x_{i}x_{j}=x_{k}x_{l}$ and $x_{i}x_{j'}=x_{k}x_{l'}$.
But this means that $x_{i}$ and $x_{k}$ have two different right lcms and this contradicts the assumption that $M$ is Garside. So, these functions are injective and since $X$ is finite they are bijective.
Assuming  $f_{i}$ not injective yields  generators with two different left lcms.
So, $S$ is well-defined and $(X, S)$ is non-degenerate and from (4) $(X, S)$ is also involutive.\\
It remains to show that $(X, S)$ is braided:
Since $M$ is Garside, $M$ is coherent and left coherent so from lemma \ref{coherence_implies_braided}, $(X,S)$ is braided.
Obviously condition (5) implies that $(X,S)$ is also square-free.
\end{proof}
It remains to establish the one-to-one correspondence and in order to that  we need the following terminology and claims.
\begin{defn}
A \textbf{tableau monoid} is a monoid $M=\operatorname{Mon} \langle X\mid R \rangle$  satisfying the conditions:
 i) $X=\{x_{1},..,x_{n}\}$.\\
ii) Each side of a relation in $R$ has length 2.
  \end{defn}
  The reason of the name is that it can be presented by a table.
 \begin{defn}
 We say that two tableau monoids $M=\operatorname{Mon} \langle X\mid R \rangle$ and $M'=\operatorname{Mon} \langle X'\mid R' \rangle$ are \textbf{t-isomorphic} if there exists a bijection $s:X \rightarrow X'$ such that $x_{i}x_{j}=x_{k}x_{l}$ is a relation in $R$ if and only if $s(x_{i})s(x_{j})=s(x_{k})s(x_{l})$ is a relation in $R'$.
 \end{defn}
 Clearly, if two tableau monoids are t-isomorphic then they are isomorphic and the definition is enlarged to groups.
 We recall that  $(X,S)$ and  $(X',S')$  are said to be \textbf{isomorphic} if there exists a bijection $\phi: X \rightarrow X'$ which maps $S$ to $S'$, that is $S'(\phi(x),\phi(y))=(\phi(S_{1}(x,y)),\phi(S_{2}(x,y)))$.
We show that if two non-degenerate, involutive and braided
 solutions  are isomorphic, then their structure groups (monoids) are t-isomorphic tableau groups (monoids) and conversely t-isomorphic Garside tableau monoids (satisfying additionally the conditions (2) and (4) from Theorem A) yield isomorphic non-degenerate, involutive and braided ``set-theoretical'' solutions.
\begin{claim}
Let $(X,S)$ and  $(X',S')$ be non-degenerate, involutive and braided ``set-theoretical'' solutions.
Assume $(X,S)$ and  $(X',S')$  are isomorphic.
Then their structure groups (monoids) $G$ and $G'$ are t-isomorphic tableau groups (monoids).
\end{claim}
\begin{proof}
Clearly, the structure groups (monoids) $G$ and $G'$ are tableau groups (monoids).
We need to show that $G$ and $G'$ are t-isomorphic.
Since $(X,S)$ and  $(X',S')$  are isomorphic, there exists a bijection $\phi: X \rightarrow X'$ which maps $S$ to $S'$, that is \\$S'(\phi(x),\phi(y))=(\phi(S_{1}(x,y)),\phi(S_{2}(x,y)))$.
So, since by definition \\$S(x,y)=(S_{1}(x,y),S_{2}(x,y))$, we have $xy=tz$ iff $\phi(x)\phi(y)=\phi(t)\phi(z)$. That is, if we take $s$ to be equal to $\phi$ we have that $G$ and $G'$ are t-isomorphic.
\end{proof}
\begin{claim}
Let $M=\operatorname{Mon} \langle X\mid R \rangle$ and $M'=\operatorname{Mon} \langle X\mid R' \rangle$ be \textbf{t-isomorphic} tableau Garside monoids each satisfying the conditions (2) and (4) from Theorem \ref{garside_struct_group}.
Let $(X,S)$ and  $(X',S')$ be the non-degenerate, involutive and braided ``set-theoretical'' solutions defined respectively by $M$ and $M'$.
Then $(X,S)$ and  $(X',S')$  are isomorphic.
\end{claim}
\begin{proof}
Take $\phi$ to be equal to $s$ and from the definition of $S$ and $S'$ from their tableau we have
$S'(\phi(x),\phi(y))=(\phi(S_{1}(x,y)),\phi(S_{2}(x,y)))$, that is $(X,S)$ and  $(X',S')$  are isomorphic.
\end{proof}
\section{About permutation solutions which are not involutive}
Let $X$ be a set and let $S:X^{2} \rightarrow X^{2}$ be a mapping.
We recall that the permutation solution  (Lyubashenko) is defined to be (see section 2):\\
a non-degenerate, involutive and braided solutionof the form $S(x,y)=(g(y),f(x))$, where $f,g:X\rightarrow X$.\\
 It holds that $(X,S)$ is  nondegenerate iff $f,g$ are bijective,  $(X,S)$ is  braided  iff $fg=gf$ and  $(X,S)$ is involutive  iff $g=f^{-1}$.\\
We will show here
that the structure group of  not necessarily involutive permutation  solutions  are Garside.
Let $G$ be the structure group of a not necessarily involutive permutation  solution  $(X,S)$ and let $M$ be the monoid with the same presentation.
We  define an equivalence relation on the set $X$ in the following way:\\
$x\equiv x'$ if and only if there is an integer $k$ such that $(fg)^{k}(x)=x'$\\
We define $X'=X/ \equiv$ and we define
 functions $f',g': X' \rightarrow X'$ in the following way:\\
 $f'([x])=[f(x)]$ and  $g'([x])=[g(x)]$, where $[x]$ denotes the equivalence class of $x$ modulo $\equiv$.\\
 We then define $S':X'\times X'\rightarrow X'\times X'$ in the following way:\\
$S'([x],[y])= (g'([y]),f'([x])) = ([g(y)],[f(x)])$.\\
Our aim is to show that $(X',S')$ is a well-defined non-degenerate, involutive and braided solution (a permutation solution) and that its structure group $G'$ is isomorphic to $G$, but before we illustrate the main ideas of the proofs to come with an example.\\
\textbf{Example $2$}
$X=\{x_{1},x_{2},x_{3},x_{4},x_{5}\}$ and let $f=(1,4)(2,3)$ and $g=(1,2)(3,4)$.
 Then $f,g$ are bijective and satisfy $fg=gf=(1,3)(2,4)$ but $fg \neq Id$, so $(X,S)$ is a non-degenerate and braided (permutation) solution, where  $S(x,y)=(g(y),f(x))$.\\
The set of relations $R$ obtained is:\\
$\begin{array}{ccc}
x_{1}^{2}=x_{2}x_{4}=x_{3}^{2}=x_{4}x_{2}&&
x_{1}x_{2}=x_{1}x_{4}=x_{3}x_{4}=x_{3}x_{2}\\
x_{2}^{2}=x_{1}x_{3}=x_{4}^{2}=x_{3}x_{1}&&
x_{1}x_{5}=x_{5}x_{4}=x_{3}x_{5}=x_{5}x_{2}\\
x_{2}x_{1}=x_{2}x_{3}=x_{4}x_{3}=x_{4}x_{1}&&
x_{2}x_{5}=x_{5}x_{3}=x_{4}x_{5}=x_{5}x_{1}\\
\end{array}$\\
Using $\equiv$ defined above, we have $X'=\{ [x_{1}],[x_{2}],[x_{5}]\}$, with $x_{1}\equiv x_{3}$
and $x_{2}\equiv x_{4}$, since in this example it holds that $fg(1)=3$ and $fg(2)=4$.
 Applying the definition of $S'$ yields  $S'([x_{1}],[x_{1}])=([g(1)],[f(1)])=([2],[4])=([2],[2])$ and so on. So, $G'= \operatorname{Gp}\langle[x_{1}],[x_{2}],[x_{5}] \mid
[x_{1}]^{2}=[x_{2}]^{2},[x_{1}][x_{5}]=[x_{5}][x_{2}],
[x_{2}][x_{5}]=[x_{5}][x_{1}]\rangle$.
Note that in $G$, it holds that $x_{1}=x_{3}$ and $x_{2}=x_{4}$ since many of the defining relations from $R$ involve cancellation and $G= \operatorname{Gp}\langle x_{1},x_{2},x_{5} \mid x_{1}^{2}=x_{2}^{2}, x_{1}x_{5}=x_{5}x_{2},x_{2}x_{5}=x_{5}x_{1} \rangle$. We show that in the general case and before we proceed we need the following technical lemma.
\begin{lem}\label{lem:calcul_Sk}
If $k$ is odd, then \\$S^{k}(x,y)=(f^{(k-1)/2}g^{(k+1)/2}(y),f^{(k+1)/2}g^{(k-1)/2}(x))$. \\
If $k$ is even, then $S^{k}(x,y)=(f^{k/2}g^{k/2}(x),f^{k/2}g^{k/2}(y))$.
\end{lem}
\begin{proof}
The proof is by induction on $k$:\\
 If $k=1$, then $S^{1}(x,y)=(f^{(1-1)/2}g^{(1+1)/2}(y),f^{(1+1)/2}g^{(1-1)/2}(x))=(g(y),f(x)).$
 If $k=2$, then $S^{2}(x,y)=S(g(y),f(x))=(gf(x),fg(y))=(f^{2/2}g^{2/2}(x),f^{2/2}g^{2/2}(y)).$\\
 If $k=3$, then\\
 $S^{3}(x,y)=S(gf(x),fg(y))=(gfg(y),fgf(x))=(fg^{2}(y),f^{2}g(x))$.\\
 Assume $k$ is odd and assume that the formula holds for $k$, so
 we need to show that it holds  for $k+2$ also.\\
 $S^{k+2}(x,y)=S^{2}S^{k}(x,y)=\\S^{2}(f^{(k-1)/2}g^{(k+1)/2}(y),f^{(k+1)/2}g^{(k-1)/2}(x))=\\
 (gf(f^{(k-1)/2}g^{(k+1)/2}(y)),fg(f^{(k+1)/2}g^{(k-1)/2}(x))=\\
 (f^{(k+1)/2}g^{(k+3)/2}(y),f^{(k+3)/2}g^{(k+1)/2}(x))$.\\
 The same proof for $k$ even holds.
\end{proof}
\begin{lem}\label{equiv_cancel}
Let $x,x' \in X $. If $x \equiv x'$, then $x$ and $x'$ are equal in $G$.
\end{lem}
\begin{proof}
 Let  $x,x'$ be in  $X$ such that  $x \equiv x'$.
If $x \equiv x'$ then it means  that there is an integer $k$ such that
  $(fg)^{k}(x)=f^{k}g^{k}(x)=x'$.
 If $k$ is odd, then let $y$ in $X$ be defined in the following way:
  $y=f^{(k+1)/2}g^{(k-1)/2}(x)$.
  Then from lemma \ref{lem:calcul_Sk}, \\
   $S^{k}(x,y)=(f^{(k-1)/2}g^{(k+1)/2}(y),f^{(k+1)/2}g^{(k-1)/2}(x))\\
                  =(f^{(k-1)/2}g^{(k+1)/2}f^{(k+1)/2}g^{(k-1)/2}(x),y)=((fg)^{k}(x),y)=(x',y)$\\
  So, there is a relation $xy=x'y$ in $G$ which implies that $x=x'$ in $G$.\\
  If $k$ is even and $(fg)^{k}(x)=x'$, then there is an element  $x'' \in X$ such that
   $(fg)^{k-1}(x)=x''$,  where $k-1$ is odd. So, from the subcase studied above,
    there is $y \in X$,  $y=f^{k/2}g^{(k-2)/2}(x)$,  such that there is a relation $xy=x''y$ in $G$ and this
     implies that $x=x''$ in $G$.
Additionally, it holds that $(fg)(x'')=x'$, so from the same argument as above there is $z \in X$  such that there is  a relation $x'z=x''z$ in $G$
and this implies that $x'=x''$ in $G$. So, $x=x'$ in $G$.\\
\end{proof}
We now show that $(X',S')$ is a a well-defined non-degenerate, involutive and braided solution and this  implies from Theorem \ref{theo:garside} that its structure group $G'$ is Garside.
\begin{lem}
The functions $f'$ and $g'$ are well defined.
Hence $S'$ is well-defined.
\end{lem}
\begin{proof}
Let $x,x'$ be in $X$ such that $x\equiv x'$.
We will show that $f'([x']) = f'([x])$, that is $[f(x')]=[f(x)]$.
Since  $x\equiv x'$, there is an integer $k$ such that $x'=(fg)^{k}(x)$, so $f'([x'])=[f(x')]=[f(fg)^{k}(x)]=[(fg)^{k}f(x)]=[f(x)]=f'([x])$, using the fact that $fg=gf$.
The same proof holds for $g'([x'])=g'([x])$.
\end{proof}
\begin{lem}
The functions $f'$ and $g'$ are bijective.
Hence $S'$ is non-degenerate.
\end{lem}
\begin{proof}
Assume that there are $x,y \in X$ such that $f'([x])=f'([y])$, that is $[f(x)]=[f(y)]$.
By the definition of $\equiv$, this means that there is an integer $k$ such that $f(x)=(fg)^{k}f(y)$, that is $f(x)=f(fg)^{k}(y)$,  since $fg=gf$. But $f$ is bijective, so $x=(fg)^{k}(y)$, which means that $[x]=[y]$. The same proof holds for $g'$, using the fact that $g$ is bijective.
\end{proof}
\begin{lem}
(1) $S'$ is braided, that is $f'g'=g'f'$.\\
(2) $S'$ is involutive, that is $f'g'=g'f'=id_{X'}$.
\end{lem}
\begin{proof}
Let $[x]\in X'$.\\
(1) From  the definition of $f',g'$, we have $f'(g'([x]))=f'([g(x)])=[f(g(x))]$  and $g'f'([x])=g'([f(x)])=[g(f(x))]$. Since $fg=gf$, we get  $f'g'=g'f'$.\\
(2) From the definition of $\equiv$, we have $[fg(x)]=[x]$, so $f'g'([x])=[x]$, that is $f'g'=id_{X'}$.
\end{proof}
\begin{lem} \label{lem:G'isGarside}
Let $(X,S)$  be a not necessarily involutive permutation  solution. Let $X'=X/ \equiv$ and let $G'$ be the structure group corresponding to  $(X',S')$, as defined above.
Then $G'$ is  Garside.\\
Furthermore,  if $x_{i}x_{j}=x_{k}x_{l}$  is  a defining relation in $G$, then $[x_{i}][x_{j}]=[x_{k}][x_{l}]$ is a defining relation in $G'$.
\end{lem}
\begin{proof}
It holds that $(X',S')$ is a non-degenerate, braided and involutive permutation  solution,
 so by Theorem \ref{theo:garside} the group $G'$ is a Garside group.\\
 Assume that  $x_{i}x_{j}=x_{k}x_{l}$  is  a defining relation in $G$, that is $S(x_{i},x_{j})=(x_{k},x_{l})$.
 From the definition of $S'$, we have  $S'([x_{i}],[x_{j}])=([g(x_{j})],[f(x_{i})])\\=([x_{k}],[x_{l}])$, that is there is a defining relation $[x_{i}][x_{j}]=[x_{k}][x_{l}]$ in $G'$. Note that this relation  may be a trivial one if
 $[x_{i}]=[x_{k}]$ and $ [x_{j}]=[x_{l}]$ in $G'$
\end{proof}
Now, it remains to show that the structure group $G$ is isomorphic
to  the group $G'$.
\begin{thm}
Let $G$ be the structure group of a not necessarily involutive permutation  solution  $(X,S)$.
Then the group $G$ is  Garside.
\end{thm}
\begin{proof}
We will show that the group $G=\langle X \mid R\rangle  $ is isomorphic to the group  $G'$, where $G'$ is the structure group of $(X',S')$ and
$X'=X/ \equiv$, and from lemma \ref{lem:G'isGarside} we obtain that the group $G$ is a Garside group.\\
 Let define the following quotient  map $\Phi: X \rightarrow X'$  such that for  $x \in X$, $\Phi (x)=[x]$.
 From lemma \ref{lem:G'isGarside}, $\Phi:G \rightarrow G'$ is an homomorphism of groups, so $\Phi $ is an epimorphism.
We need to show that $\Phi$ is injective:
We show that  if $[x][y]=[t][z]$ is a non-trivial defining relation in $G'$, then $xy=tz$  is  a defining relation in $G$.
If $[x][y]=[t][z]$ is a non-trivial defining relation in $G'$, then since
 $S'([x],[y])= ([g(y)],[f(x)])$, we have that $[g(y)]=[t]$ and $[f(x)]=[z]$. That is, there are
$z' \in [z]$ and $t' \in [t]$ such that $g(y)=t'$ and $f(x)=z'$.
 This implies that $S(x,y)=(g(y),f(x))=(t',z')$, that is $xy=t'z'$  is  a defining relation in $G$.
 It holds that $t \equiv t'$ and $z \equiv z'$, so from lemma \ref{equiv_cancel}, $t=t'$ and $z=z'$ in $G$.
 So, $xy=tz$  is  a defining relation in $G$.
 Note that if $[x][y]=[t][z]$ is a trivial relation in $G'$, that is $[x]=[t]$ and $[y]=[z]$, then
 from lemma \ref{equiv_cancel} $x=t$ and $y=z$ in $G$ and so $xy=tz$ holds trivially in $G$.
So, $\Phi$ is an isomorphism of the groups $G$ and $G'$ and from lemma \ref{lem:G'isGarside}, we have that $G$ is Garside.
\end{proof}

\section{Calculation of $\Delta$ for a permutation solution}
Let $(X,S)$ be a permutation solution. In case $(X,S)$ is not involutive, the results  obtained here can be applied using the procedure described in Section $9$.
We denote by $G$ the structure group of $(X,S)$ and by $M$ the monoid with the same presentation.
In order to calculate a Garside element $\Delta$, we need the following technical lemmas.  In \cite{gateva}, results of the same flavor (with no proof) are obtained for $(X,S)$ a square-free, non-degenerate, involutive and  braided  solution. So, it seems that these are well-known facts for square-free  solutions.\\
We recall that $(X,S)$ is given by  $X=\{x_{1},..,x_{n}\}$, $S(i,j)=(f(j),f^{-1}(i))$, where  $f$ is a permutation on $\{1,..,n\}$. $f$ is described as the product of disjoint cycles of the form \\ $f=(t_{1,1},..,t_{1,m_{1}})(t_{2,1},..,t_{2,m_{2}})
(t_{k,1},..,t_{k,m_{k}})(s_{1})..(s_{l})$, where all  $t_{i,j},s_{*} \in \{1,..,n\}$.
\begin{lem}\label{same_cycle}
It holds for each $1 \leq i \leq  k $ that $x_{t_{i,1}}^{m_{i}}=x_{t_{i,2}}^{m_{i}}=..=x_{t_{i,m_{i}}}^{m_{i}}$.
\end{lem}
\begin{proof}
From  $S(t_{i,p},t_{i,p})=(t_{i,p+1},t_{i,p-1})$, we have
$x_{t_{i,p}}^{m_{i}}=x_{t_{i,p}}^{m_{i}-2}x_{t_{i,p}}x_{t_{i,p}}=
x_{t_{i,p}}^{m_{i}-3}x_{t_{i,p}}x_{t_{i,p+1}}x_{t_{i,p-1}}=
x_{t_{i,p}}^{m_{i}-3}x_{t_{i,p+2}}x_{t_{i,p-1}}x_{t_{i,p-1}}$.\\
We go on inductively and we obtain
$x_{t_{i,p}}^{m_{i}}=x_{t_{i,p-1+m_{i}}}x_{t_{i,p-1}}^{m_{i}-1}$.\\
Since $p-1+m_{i}=p-1 (mod m_{i})$, we  have $x_{t_{i,p}}^{m_{i}}=x_{t_{i,p-1}}^{m_{i}}$.
\end{proof}
From lemma \ref{same_cycle}, there is no ambiguity if we  denote by $x_{t_{i}}^{m_{i}}$ any element $x_{t_{i,*}}^{m_{i}}$, since all these are equal. Moreover, for each cycle of length $m_{i}$ the calculations  are modulo $m_{i}$, so we often omit the modulo $m_{i}$.
\begin{lem}\label{one_dft_cycles}
It holds for  $1 \leq i,j \leq k $ that\\ $x_{t_{i,*}}x_{t_{j}}^{m_{j}}=x_{t_{j}}^{m_{j}}x_{t_{i,*-m_{j}}}$ and
$x_{t_{i}}^{m_{i}}x_{t_{j,*}}=x_{t_{j,*+{m_{i}}}}x_{t_{i}}^{m_{i}}$.
\end{lem}
\begin{proof}
$x_{t_{i,*}}x_{t_{j}}^{m_{j}}=x_{t_{i,*}}x_{t_{j,p}}x_{t_{j,p}}^{m_{j}-1}=
x_{t_{j,p+1}}x_{t_{i,*-1}}x_{t_{j,p}}x_{t_{j,p}}^{m_{j}-2}\\=..
=x_{t_{j,p+1}}^{m_{j}}x_{t_{i,*-m_{j}}}$.
From lemma  \ref{same_cycle}, we obtain
$x_{t_{i,*}}x_{t_{j}}^{m_{j}}=x_{t_{j}}^{m_{j}}x_{t_{i,*-m_{j}}}$.\\
The second statement is deduced from the first one.
\end{proof}
\begin{lem}\label{dft_cycles}
It holds for  $1 \leq i,j \leq k $ that $x_{t_{i}}^{m_{i}}x_{t_{j}}^{m_{j}}=x_{t_{j}}^{m_{j}}x_{t_{i}}^{m_{i}}$.
\end{lem}
\begin{proof}
From lemma \ref{one_dft_cycles}, we have that
$x_{t_{i,*}}^{m_{i}}x_{t_{j}}^{m_{j}}=
x_{t_{j}}^{m_{j}}x_{t_{i,*-m_{j}}}^{m_{i}}$.\\ So, we obtain
$x_{t_{i}}^{m_{i}}x_{t_{j}}^{m_{j}}=x_{t_{j}}^{m_{j}}x_{t_{i}}^{m_{i}}$, since from lemma \ref{same_cycle}, $x_{t_{i,*-m_{j}}}^{m_{i}}=x_{t_{i}}^{m_{i}}$.
\end{proof}
\begin{lem}\label{stable_point_commute}
It holds that $x_{s_{j}}x_{t_{i}}^{m_{i}}=x_{t_{i}}^{m_{i}}x_{s_{j}}$, for each $1 \leq j \leq l $ and $1 \leq i \leq k$.
\end{lem}
\begin{proof}
From  $S(t_{i,p},s_{j})=(s_{j},t_{i,p-1})$, we have
$x_{t_{i,p}}x_{s_{j}}=x_{s_{j}}x_{t_{i,p-1}}$.\\
So, inductively we have $x_{s_{j}}x_{t_{i}}^{m_{i}}=x_{t_{i}}^{m_{i}}x_{s_{j}}$, using lemma \ref{same_cycle}.
\end{proof}
\begin{prop}\label{delta_permut}
The element $\Delta= x_{t_{1}}^{m_{1}}x_{t_{2}}^{m_{2}}..x_{t_{k}}^{m_{k}}x_{s_{1}}..x_{s_{l}}$
is a Garside element.
\end{prop}
\begin{proof}
In order to show that $\Delta$ is a Garside element, we need to show: \\ (1) the left divisors of $\Delta$ are the same as the right divisors.\\
(2) there is a finite number of them.\\
(3) they generate $M$.\\
Condition (2) holds obviously.
From the precedent lemmas, it results that all the elements from $X$ are left and right divisors of $\Delta$, so condition (3) holds.
So, it remains to show that condition (1) holds.\\
Any element which is the product in any order of elements of the form $x_{t_{i}}^{m_{i}}$  and $x_{s_{j}}$ for any $i,j$ is a left and a right divisor, from the commutativity relations obtained in lemma \ref{dft_cycles}. The problems arise with elements of the form
$x_{t_{i_{1}}}^{m_{i_{1}}}..x_{t_{i_{o}}}^{m_{i_{o}}}x_{t_{i_{p}}}^{q}$ or $x_{s_{j}}x_{t_{i_{p}}}^{q}$ ($q < m_{i_{p}}$) which are left divisors but it is not clear at all that they are right divisors and with elements of the form $x_{t_{i_{1}}}^{q}x_{t_{i_{2}}}^{m_{i_{2}}}..x_{t_{i_{p}}}^{m_{i_{p}}}$ or $x_{t_{i_{1}}}^{q}x_{s_{j}}$ ($q < m_{i_{p}}$)
which are right divisors but it is not clear at all that they are left divisors.
We solve the problem with elements of the form $x_{t_{i_{1}}}^{m_{i_{1}}}..x_{t_{i_{o}}}^{m_{i_{o}}}x_{t_{i_{p}}}^{q}$ and $x_{s_{j}}x_{t_{i_{p}}}^{q}$ ($q < m_{i_{p}}$) and the same kind of proof holds for the second type of elements. It is sufficient to look at an element of the form
$x_{t_{i}}^{m_{i}}x_{t_{j}}^{q}$ and inductively obtain the result. We will show that this element is also a right divisor of $\Delta$. It holds that $\Delta$ is equal to an  element $ux_{t_{i}}^{m_{i}}x_{t_{j}}^{m_{j}}$, for some $u$.\\
So, $ux_{t_{i}}^{m_{i}}x_{t_{j}}^{m_{j}}=ux_{t_{i}}^{m_{i}}x_{t_{j,*}}^{m_{j}-q}x_{t_{j,*}}^{q}=
ux_{t_{j,*+m_{j}-q}}^{m_{j}-q}x_{t_{i}}^{m_{i}}x_{t_{j,*}}^{q}$, \\ using lemma \ref{one_dft_cycles} $m_{j}-q$ times.
So, $x_{t_{i}}^{m_{i}}x_{t_{j}}^{q}$ is a right divisor of $\Delta$.\\
 Now, we show that $x_{s_{j}}x_{t_{i}}^{q}$ ($q < m_{i}$) is a right divisor of $\Delta$.
It holds that $\Delta$ is equal to the element $ux_{s_{j}}x_{t_{i}}^{m_{i}}$, for some $u$.
So, $ux_{s_{j}}x_{t_{i}}^{m_{i}}=
ux_{s_{j}}x_{t_{i,*}}^{m_{i}-q}x_{t_{i,*}}^{q}=
ux_{t_{i,*+m_{i}-q}}^{m_{i}-q}x_{s_{j}}x_{t_{i,*}}^{q}$, using lemma \ref{stable_point_commute} $m_{i}-q$ times.
So, $x_{s_{j}}x_{t_{i}}^{q}$  is a right divisor of $\Delta$.
\end{proof}
\textbf{Example 2 cont':}
Let $X=\{x_{1},x_{2},x_{3}\}$ and $S(x_{i},x_{j})=(f(j),f^{-1}(i))$, where $f=(1,2,3)$, be a permutation solution.
Then it holds that $x_{1}^{3}=x_{2}^{3}=x_{3}^{3}$ in $M$ and each word $x_{i}^{3}$ for $i=1,2,3$ represents a Garside element, denoted by $\Delta$. \\The set of simples is $\chi=\{\epsilon, x_{1},x_{2},x_{3},x_{1}^{2},x_{2}^{2},x_{3}^{2},\Delta\}$. The exponent of $M$ (see section $3.3$) is equal to 1, since the function $x \rightarrow (x\setminus \Delta)\setminus \Delta$ from $\chi$ to itself is the identity.
As an example, the image of $x_{1}$ is $(x_{1}\setminus \Delta)\setminus \Delta=x_{1}^{2}\setminus \Delta=x_{1}$,
the image of $x_{2}^{2}$ is $(x_{2}^{2}\setminus \Delta)\setminus \Delta=x_{2}\setminus \Delta=x_{2}^{2}$ and so on.
So, from Theorem \ref{thm_deltapure_indecomp} and Proposition \ref{center_cyclic}, the center of the structure group of $(X,S)$ is cyclic and generated by $\Delta$.

\section{Appendix: Proof of Claim \ref{M_right_cancel}: $M$ is right cancellative}
In \cite{gateva+van}, using a different approach, the authors show that the monoids of left I-type are cancellative.
In order to make the paper self-contained, we give in this appendix our proof.
As we shall see, the proof of the right cancellativity is a 'rewriting ' of the
 proof from \cite{garside}, with the relevant alterations. Let
$w_{1},w_{2}$ be words in the free monoid generated by $X$. If
$w_{1}$ is transformed into  $w_{2}$ by a sequence of $t$ single
applications of the defining relations, then the whole
transformation will be said to be of \textbf{chain-length} $t$.
The right cancellativity of $M$  will be proved in a direct proof
by induction on the length of the words and the chain-length of
transformations, but before we need the following lemma.

\begin{proof}
Assume that $w_{1}a=w_{2}a$ in $M$, with $w_{1},w_{2}$ words in the free monoid
generated by $X$ and $a \in X$.
Since the defining relations are length-preserving, $w_{1}$ and $w_{2}$ have the same length.
We will show that $w_{1}=w_{2}$ by induction on the length of   $w_{1}$ (and $w_{2}$) and on the
 chain-length of transformations.\\
If $w_{1}=w_{2}=\epsilon$, then there is nothing to prove.\\
Assume that $w_{1},w_{2} \in X$, that is there are $1 \leq k,l
\leq n$ such that $w_{1}=x_{k}$  and $w_{2}=x_{l}$. We will show
that necessarily  $x_{k}=x_{l}$.
 From claim \ref{cl_compl}, we have  that $x_{k}a=x_{l}a$ implies $x_{k}=x_{l}$.\\
Now, assume that $w_{1}a=w_{2}a$ in $M$ implies  $w_{1}=w_{2}$ in $M$ when  :\\
(*) the length of $w_{1}$ is less or equal to   $m-1$ and the transformations  have any chain-length.\\
and\\
(**) the length of $w_{1}$ is greater or equal  to  $m$ and  the chain-length is less or equal to $q-1$.\\
Assume that $w_{1}$ and  $w_{2}$ have length $m$ and  that
$w_{1}a=w_{2}a$ in $M$ through a transformation of chain-length $q$. Let
the successive words of the transformation be $u_{1}=w_{1}a$,
$u_{2}=..$,..,$u_{q+1}=w_{2}a$. Choose arbitrarily any
intermediate word $u_{p}$ in the chain, so $u_{p}=wa'$ in $M$ for some
word $w$ and $a' \in X$. So,  $w_{1}a=wa'=w_{2}a$ in $M$ and each
transformation  $w_{1}a=wa'$ and $wa'=w_{2}a$  is of chain-length
less than $q$. If $a=a'$ then the induction hypothesis can  be
applied, i.e  $w_{1}=w$ and $ w=w_{2}$, so $w_{1}=w_{2}$. Assume
$a \neq a'$. Then from lemma \ref{lem_cancel} there is only one
defining relation of the form $x_{i} a= x_{j} a'$, where $x_{i},
x_{j}$ are in $X$. So, we have that  $w_{1}=Px_{i}$ and
$w=Px_{j}$ in $M$. On the other hand, we have that  $w=Qx_{j}$ and
$w_{2}=Qx_{i}$ in $M$. So, $w=Qx_{j}=Px_{j}$in $M$ and by induction hypothesis
this implies $P=Q$ and then $w_{1}=w_{2}$ in $M$. That is $M$ is right cancellative.
\end{proof}


\begin{thebibliography}{20}
\bibitem{charney} R.Charney, J.Meier and K.Whittlesey, Bestvina's normal form complex and the homology of Garside groups, Geo.Dedicata 105 (2004), p.171-188.
    \bibitem{Clifford} A.H.Clifford, G.B.Preston, The algebraic theory of semigroups. Mathematical Surveys, No. 7 AMS, Providence,R.I.1967.
\bibitem{deh_francais} P.Dehornoy, Groupes de Garside, Ann.Scient.Ec.Norm.Sup.35 (2002),p.267-306.
\bibitem{deh_livre} P.Dehornoy, Braids and Self-Distributivity, Progress in Math.192, Birkhauser (2000).
\bibitem{deh_torsion}P.Dehornoy, Gaussian groups are torsion free,
 J. of Algebra 210 (1998),p.291-297.
 \bibitem{deh_complte}P.Dehornoy, Complete positive group presentations, J.Algebra 268 (2003), p.156-197.
\bibitem{deh_homologie} P.Dehornoy, Y.Lafont,  Homology of Gaussian groups, Ann.Inst.Fourier(Grenoble)53(2003),
 no. 2, p.489-540.
\bibitem{deh_Paris} Gaussian groups and Garside groups: two generalizations of Artin groups, Proc.London Math.Soc.79(1999), p.569-604.
\bibitem{drinf} V.G.Drinfeld, On some unsolved problems in quantum group theory, Lec.Notes Math.1510 (1992), p.1-8.
    \bibitem{etingof} P.Etingof, T.Schedler, A.Soloviev, Set-Theoretical
Solutions to the Quantum Yang-Baxter equation, Duke Math. J. 100
(1999), no. 2, p.169-209.
\bibitem{garside} F.A.Garside, The braid group and other groups,
Quart.J.Math.Oxford 20 N.78(1969),p.235-254.
\bibitem{gateva} T.Gateva-Ivanova, A combinatorial approach to the set-theoric solutions of the Yang–Baxter equation, J.Math.Phys.45(10)(2005), p.3828–3858.

\bibitem{gateva+van}T.Gateva-Ivanova, M.Van den Bergh, Semigroups of $I$-type, J. Algebra 206 (1998), No.1, p.97-112.
\bibitem{jespers+okninski}E.Jespers, J.Okninski, Monoids and groups of I-type, Algebras and Representation Theory 8(2005), p.709-729.
\bibitem{picantin} M.Picantin, The center of thin Gaussian groups,
J.Algebra 245 (2001), p.92-122.
\bibitem{rump} W.Rump, A decomposition theorem for square-free unitary solutions of the quantum Yang-Baxter equation, Advances in Math. 193 (2005), p.40-55.
\end{thebibliography}
\end{document}